\documentclass[a4paper,fleqn]{cas-sc}

\usepackage[numbers,sort&compress]{natbib}
\usepackage{amsmath,amsfonts,amssymb}
\usepackage{graphicx}
\usepackage{tikz}
\usetikzlibrary{arrows.meta,calc}
\usepackage{float}
\usepackage{placeins}
\usepackage{booktabs}
\usepackage{algorithm}
\usepackage{algpseudocode}
\usepackage{comment}
\usepackage{array,longtable}

\hypersetup{hypertexnames=false}

\newtheorem{theorem}{Theorem}[section]
\newtheorem{lemma}[theorem]{Lemma}
\newtheorem{proposition}[theorem]{Proposition}
\newtheorem{corollary}[theorem]{Corollary}
\newdefinition{definition}[theorem]{Definition}
\newdefinition{remark}[theorem]{Remark}
\newdefinition{assumption}[theorem]{Assumption}
\newproof{proof}{Proof}

\AtBeginDocument{}

\begin{document}

\let\WriteBookmarks\relax
\def\floatpagepagefraction{1}
\def\textpagefraction{.001}

\shorttitle{Domain-decomposed random-feature evolution for pressure diffusion}
\shortauthors{P. Li et al.}

\title[mode=title]{Domain-decomposed Evolutional Deep Neural Network with
Random Features for Transient Pressure Diffusion with Discontinuous and
High-Contrast Coefficients}

\author[1,2]{Peiqi Li}[orcid=0009-0005-9789-0509]
\ead{Peiqi.Li23@student.xjtlu.edu.cn}
\credit{Methodology, Project administration, Validation, Writing--original draft}

\author[1]{Jie Chen}[orcid=0000-0003-1036-007X]
\cormark[1]
\ead{Jie.Chen01@xjtlu.edu.cn}
\credit{Conceptualization, Supervision, Validation, Writing--review and editing}

\author[1]{Hui Zhang}[orcid=0000-0001-7245-0674]
\ead{Hui.Zhang@xjtlu.edu.cn}
\credit{Methodology, Supervision, Writing--review and editing}

\author[2]{Simon Hands}[orcid=0000-0001-5720-7852]
\ead{Simon.Hands@liverpool.ac.uk}
\credit{Validation, Supervision}

\affiliation[1]{organization={Department of Applied Mathematics,
Xi'an Jiaotong-Liverpool University},
city={Suzhou},
postcode={215123},
country={China}}

\affiliation[2]{organization={Department of Mathematical Sciences,
University of Liverpool},
city={Liverpool},
postcode={L69 3BX},
country={United Kingdom}}

\cortext[cor1]{Corresponding author.}

\begin{abstract}
Transient pressure diffusion in heterogeneous porous media becomes difficult to resolve efficiently when permeability is discontinuous and spans several orders of magnitude. In this paper, we develop a domain-decomposed random-feature evolutional deep neural network (DD RF-EDNN) that separates spatial approximation from temporal evolution. Permeability-informed random features are compressed into an orthonormal pressure space, and a conservative finite-volume operator is projected onto this space so that only the reduced coordinates are advanced in time. This formulation retains the dissipative structure of the discrete flow
problem without repeated neural-network optimization. The analysis quantifies dictionary and singular-value truncation errors, proves contractivity under an energy-consistent dissipativity correction, and derives a conditional grid-level estimate for homogeneous dynamics and time-independent data admitting a steady lifting. Numerical experiments on low-permeability inclusions, high-conductivity blocks and channels, and three-dimensional grids assess the accuracy of the method across distinct coefficient configurations. The method reaches a final-time relative $L^2$ error of $9.54\times10^{-4}$ in the principal benchmark and maintains errors on the order of $10^{-2}$ at a permeability contrast of $10^3$. Numerical diagnostics further show that the observed error balance and temporal convergence are consistent with the analysis. These results
support DD RF-EDNN as a structure-preserving and interpretable reduced
formulation for transient simulations on prescribed heterogeneous media.
\end{abstract}


\begin{keywords}
Transient Darcy flow \sep Random feature method \sep Domain decomposition \sep
Evolutional deep neural network \sep Reduced-order model
\end{keywords}

\maketitle

\section{Introduction}

Transient flow through heterogeneous porous media is central to groundwater management, reservoir simulation, and subsurface energy and environmental applications. At the Darcy scale, the pressure satisfies a diffusion equation whose storage and permeability coefficients determine the temporal and spatial response. Geological permeability fields are often discontinuous and may vary by several orders of magnitude. Low-permeability inclusions delay pressure propagation, whereas conductive channels redirect flux over long distances. These features produce strongly nonuniform fluxes and widely separated local diffusion time scales, making repeated transient simulation expensive, particularly in three dimensions and in many-query studies.

Finite-difference (FVM), finite-element (FEM), and finite-volume methods (FDM) provide reliable full-order discretizations of this problem. Finite-volume schemes are especially useful because they preserve local conservation and accommodate heterogeneous coefficients and general grids \cite{eymard2000finite,aavatsmark2002introduction}. Multiscale finite-element (MsFEM) and finite-volume methods (MsFVM) reduce spatial complexity through operator-adapted local bases or coarse transmissibilities \cite{hou1997multiscale,jenny2003multi,hou2009multiscale}. Projection-based reduced-order models instead evolve coordinates in a low-dimensional trial space \cite{benner2015survey}. Proper orthogonal decomposition (POD) and reduced-basis methods can be highly efficient when representative solution snapshots are available \cite{sirovich1987turbulence,grepl2005posteriori,hesthaven2015certified,quarteroni2015reduced,haasdonk2008reduced}. Their offline construction, however, generally requires full-order solves, and substantial changes in the permeability configuration may require a new snapshot ensemble. For a prescribed heterogeneous medium, it is therefore desirable to construct a compact pressure space directly from the material field and the discrete operator, before any transient solution is computed.

Neural PDE methods construct spatial and spatiatemporal approximation spaces. Physics-informed neural networks (PINNs) and its variations  determine a space--time approximation by minimizing residual, initial-condition, and boundary-condition losses, often through a large nonconvex optimization problem \cite{raissi2019physics,jagtap2020extended,karniadakis2021physics,krishnapriyan2021characterizing}. Neural operators, such as deep operator network (DeepONet) and the Fourier neural operator (FNO), learn mappings between coefficient or state functions and solutions from an offline dataset \cite{lu2021learning,li2020fourier}. Evolutional deep neural networks (EDNNs) instead fit the initial state and then advance the representation parameters according to the governing equation \cite{du2021evolutional}. This sequential viewpoint respects the initial-value structure and avoids optimization over the complete space--time cylinder. Efficient evolution requires a spatial representation that is inexpensive to construct, well conditioned, and compatible with the discrete dynamics \cite{halder2025physics}.

Fixed random features provide a natural mechanism for constructing such a representation. A randomized feature map converts the spatial approximation into a linear problem for its output coefficients \cite{rahimi2007random,guang2006extreme}. For PDEs, random-feature spaces have been used with collocation and least-squares formulations, including global space--time and block time-marching strategies for evolution equations \cite{chen2022bridging,chen2023random}. Problem-adapted variants address anisotropic or discontinuous diffusion \cite{mei2025solving}, while Petrov--Galerkin and discontinuous Galerkin formulations provide variational or locally coupled alternatives \cite{shang2023randomized,sun2024local}. Domain-decomposed random-feature methods have also been combined with two-level solvers and overlapping Schwarz preconditioners \cite{sun2025two,shang2025overlapping}. More recent discrete-time formulations determine random-feature coefficients through a linear least-squares problem at each integration stage \cite{sun2024local,zhou2026discrete}. These developments demonstrate the flexibility of randomized trial spaces, but they leave open a distinct reduced-order question. The spatial basis should be material adapted and well conditioned, while the conservative semi-discrete operator should evolve only fixed reduced coordinates \cite{shen2026randomized}.

We address this question with a domain-decomposed random-feature evolutional deep neural network (DD RF-EDNN) for transient Darcy flow with discontinuous and high-contrast permeability. The method first allocates randomized spatial features according to material subregions and high-contrast interfaces. A two-stage singular-value decomposition removes redundancy and produces an orthogonal pressure basis. The conservative finite-volume operator is then projected onto this basis, yielding a low-dimensional evolution equation for time-dependent coordinates. The resulting workflow requires no full-order solution snapshots for basis construction and no repeated neural-network optimization during time integration. It also retains a direct connection to the discrete operator, which permits stability analysis and a transparent decomposition of the principal error sources.

The main contributions are threefold: 
\begin{enumerate}
	\item First, the material-aware domain decomposition converts permeability information into a randomized trial space and assigns additional approximation capacity near internal interfaces. A deterministic estimate separates dictionary approximation from the two singular-value truncation stages.
	\item Second, projection of the finite-volume generator yields a dissipative reduced system, while a logarithmic-norm correction restores Euclidean contractivity if numerical truncation perturbs this structure. The resulting conditional grid-level estimate separates finite-volume consistency, DD-RF approximation, dynamical closure, algebraic perturbation, and time integration; a steady-lifting corollary covers time-independent affine boundary data.
	\item Third, two- and three-dimensional experiments examine accuracy, robustness to coefficient geometry and contrast, component-level ablations, error decomposition, and computational scaling.
\end{enumerate}  

The remainder of the paper is organized as follows. Section 2 introduces the governing problem, the material-adapted random-feature construction, the orthogonal compression, and the reduced evolution. Section 3 presents the stability and error analysis. Section 4 reports the numerical experiments and diagnostic error decomposition. Section 5 discusses the implications and limitations of the method, and Section 6 concludes the paper.

\section{Methodology}\label{sec:method}
We propose the DD RF-EDNN for transient Darcy flow with discontinuous permeability. The method constructs a material-adapted random-feature space and projects the finite-volume Darcy dynamics onto a low-dimensional orthogonal pressure basis, yielding a compact reduced model for stable time evolution and pressure reconstruction. Figure~\ref{fig:domain_decomposition_general} illustrates the generic two-dimensional material partition and the corresponding material- and interface-adapted random-feature construction.

\begin{figure}[pos=htbp]
	\centering
	\begin{tikzpicture}[
		scale=0.82,
		line join=round,
		line cap=round,
		>=Latex
		]
		
		\begin{scope}[xshift=0cm]
			
			\draw[thick] (0,0) rectangle (7,7);
			
			\fill[gray!25]
			plot[smooth cycle, tension=0.8] coordinates {
				(1.1,4.8)
				(1.5,5.8)
				(2.5,6.0)
				(3.0,5.3)
				(2.7,4.5)
				(1.8,4.3)
			};
			\draw[thick]
			plot[smooth cycle, tension=0.8] coordinates {
				(1.1,4.8)
				(1.5,5.8)
				(2.5,6.0)
				(3.0,5.3)
				(2.7,4.5)
				(1.8,4.3)
			};
			
			\fill[gray!25] (4.6,5.2) rectangle (5.5,6.0);
			\draw[thick] (4.6,5.2) rectangle (5.5,6.0);
			
			\fill[gray!25] (5.7,4.0) rectangle (6.3,4.7);
			\draw[thick] (5.7,4.0) rectangle (6.3,4.7);
			
			\fill[gray!25]
			(1.0,1.1)
			-- (1.0,1.7)
			-- (2.1,1.7)
			-- (2.1,2.3)
			-- (3.4,2.3)
			-- (3.4,2.9)
			-- (4.4,2.9)
			-- (4.4,2.3)
			-- (3.8,2.3)
			-- (3.8,1.7)
			-- (2.7,1.7)
			-- (2.7,1.1)
			-- cycle;
			
			\draw[thick]
			(1.0,1.1)
			-- (1.0,1.7)
			-- (2.1,1.7)
			-- (2.1,2.3)
			-- (3.4,2.3)
			-- (3.4,2.9)
			-- (4.4,2.9)
			-- (4.4,2.3)
			-- (3.8,2.3)
			-- (3.8,1.7)
			-- (2.7,1.7)
			-- (2.7,1.1)
			-- cycle;
			
			\node at (3.5,3.7) {$\Omega_m$};
			
			\node at (2.0,5.2) {$\Omega_n$};
			
			\node at (5.05,5.6) {$\Omega_n$};
			
			\node at (0.55,6.55) {$\Omega$};
			
			\draw[-{Latex[length=2mm]}, thick]
			(3.55,5.65) -- (2.85,5.55);
			\node[right] at (3.55,5.65) {$\Gamma$};
			
			\draw[-{Latex[length=2mm]}, thick]
			(5.8,7.55) -- (5.8,7.0);
			\node[above] at (5.8,7.55) {$\partial\Omega$};
			
			\node[font=\small] at (3.5,-0.55)
			{\textbf{(a)} General material partition};
			
		\end{scope}

		\begin{scope}[xshift=9.5cm]
			
			\draw[thick] (0,0) rectangle (7,7);
			
			\fill[white] (0,3.5) rectangle (7,7);
			\fill[gray!12] (0,0) rectangle (7,3.5);
			
			\fill[gray!30] (0,3.0) rectangle (7,4.0);
			
			\draw[very thick] (0,3.5) -- (7,3.5);
			
			\draw[thick] (0,0) rectangle (7,7);
			
			\node at (1.1,6.2) {$\Omega_m$};
			\node at (1.1,0.8) {$\Omega_n$};
			
			\draw[-{Latex[length=2mm]}, thick]
			(6.35,4.75) -- (5.5,3.55);
			\node[right] at (6.35,4.75) {$\Gamma$};
			
			\draw[-{Latex[length=2mm]}, thick]
			(6.35,2.35) -- (5.5,3.05);
			\node[right] at (5.95,1.89) {$\mathcal{N}_\Gamma$};

			
			\foreach \x/\y in {
				1.4/5.3,
				2.4/6.0,
				3.4/5.1,
				4.5/6.1,
				5.6/5.3,
				2.0/4.5,
				4.0/4.6
			}{
				\draw[thick] (\x,\y) circle (0.09);
			}
			
			\foreach \x/\y in {
				1.4/1.6,
				2.5/0.9,
				3.4/1.8,
				4.6/1.0,
				5.6/1.8,
				2.0/2.5,
				4.2/2.4
			}{
				\draw[thick]
				(\x-0.08,\y-0.08) -- (\x+0.08,\y+0.08);
				\draw[thick]
				(\x-0.08,\y+0.08) -- (\x+0.08,\y-0.08);
			}

			\foreach \x in {0.8,1.6,2.4,3.2,4.0,4.8,5.6,6.4}{
				\fill (\x,3.25) circle (0.055);
				\fill (\x,3.75) circle (0.055);
			}

			\draw[-{Latex[length=2mm]}, thick]
			(6.1,6.2) -- (5.65,5.35);
			\node[right, align=left] at (6.1,6.2)
			{$V_m$};
			
			\draw[-{Latex[length=2mm]}, thick]
			(6.1,0.7) -- (5.65,1.75);
			\node[right, align=left] at (6.1,0.7)
			{$V_n$};
			
			\node[font=\small] at (3.5,-0.55)
			{\textbf{(b)} Material/interface-aware feature enrichment};
			
		\end{scope}

	\end{tikzpicture}
	
	\caption{Material-adapted random features: (a) material subregions $\{\Omega_m\}$ and interfaces $\Gamma$; (b) feature allocation within subregions and near high-contrast interfaces.}
	\label{fig:domain_decomposition_general}
\end{figure}

\subsection{Model Problem and Finite-volume Discretization}
Let $\Omega \subset \mathbb{R}^d,~d=2,3$ be the computational domain. Consider the transient Darcy-type diffusion equation
\begin{equation}\label{eq:PDE}
	S(x) \frac{\partial p}{\partial t} = \nabla \cdot \left( A(x) \nabla p \right),\quad x \in \Omega,~ t \in (0, T]
\end{equation}
subject to appropriate initial and boundary conditions. Here $p(x,t)$ denotes the pressure, $S(x)>0$ is the storage coefficient, and $A(x)$ is the permeability coefficient, which may be scalar- or tensor-valued. We assume that $A(x)$ is piecewise smooth and may have jumps across internal material interfaces.

Let the domain be decomposed into material subdomains $\Omega = \bigcup_{m=1}^{M_\Omega} \Omega_m$
with internal interfaces $\Gamma_{mn} = \partial \Omega_m \cap \partial \Omega_n$. Each $\Gamma_{mn}$ is a curve in two dimensions and a surface in three dimensions. The exact solution satisfies pressure and normal-flux continuities across each interface.

We discretize the governing equation by a finite-volume method. Let $\mathcal{T}_h$ be a finite-volume mesh with mesh size $h$, and let
\begin{equation}
	\mathbf{p}_h(t) = \left( p_1(t), p_2(t), \ldots, p_{N_h}(t) \right)^\top \in \mathbb{R}^{N_h} \notag
\end{equation}
be the vector of discrete pressure unknowns. The semi-discrete system can be written as
\begin{equation}
	M_h \frac{d \mathbf{p}_h(t)}{dt} = K_h \mathbf{p}_h(t)+\mathbf{f}_h(t),
\end{equation}
where $M_h$ is the discrete storage matrix, $K_h$ is the discrete diffusion matrix, and $\mathbf{f}_h(t)$ collects possible boundary contributions from nonhomogeneous Dirichlet or Neumann data. Defining $L_h = M_h^{-1} K_h$ and $\mathbf{b}_h(t)=M_h^{-1}\mathbf{f}_h(t)$ gives the reference semi-discrete dynamics
\begin{equation}\label{eq:FV_reference_system}
	\frac{d \mathbf{p}_h(t)}{dt} = L_h \mathbf{p}_h(t)+\mathbf{b}_h(t).
\end{equation}
For the homogeneous Dirichlet cases used in the central-inclusion and four-block tests, $\mathbf{b}_h(t)=0$. For mixed or nonhomogeneous boundary conditions, the same finite-volume operator is used and the affine boundary contribution is projected together with the reduced dynamics.

The full-order dynamics are obtained from a conservative finite-volume discretization. For a dissipative finite-volume generator, orthogonal projection onto the pressure space gives a dissipative reduced operator \cite{prusak2023optimisation}.

\subsection{Domain-decomposed Random Feature Approximation}

The pressure field is approximated by fixed spatial random features and time-dependent linear coefficients. To account for discontinuous or high-contrast permeability fields, the random feature dictionary is constructed according to a material-aware decomposition of the domain. In the implementation, this decomposition is represented by masks on the finite-volume grid. For simple benchmark geometries, the masks are prescribed from the known material layout. For irregular high-contrast fields, material masks can be constructed directly from the permeability data by thresholding the logarithmic permeability and expanding the selected high-contrast cells by a prescribed number of grid layers.

Let
\begin{equation}
	\Phi_0 = \left[ \Phi_1, \Phi_2, \ldots, \Phi_{M_{\Omega}}, \Phi_\Gamma \right] \in \mathbb{R}^{N_h \times M_0} \notag
\end{equation}
be the full random-feature matrix evaluated at the finite-volume degrees of freedom. Here $\Phi_m$ contains features associated with the material subdomain $\Omega_m$, and $\Phi_\Gamma$ contains features localized in neighborhoods of internal interfaces. The pressure approximation can be written as
\begin{equation}
	\mathbf{p}_M(t) = \Phi_0 z(t), \quad z(t) \in \mathbb{R}^{M_0}
\end{equation}

At the continuous level, this corresponds to a domain-decomposed random feature space
\begin{equation}
	V_{M_0}^{DD} = \left( \sum_{m=1}^{M_\Omega} V_m \right) + V_\Gamma
\end{equation}

The subdomain feature spaces $V_m$ provide local approximation capacity inside different regions, while $V_\Gamma$ enriches the representation near interfaces. This algebraic representation is dimension-independent.

A typical random feature takes the form
\begin{equation}
	\Phi_j(x) = b_\Omega (x) \omega_j(x) \sigma (w_j \cdot x + \beta_j)
\end{equation} 
where $b_\Omega (x)$ is a boundary-compatible factor, $\omega_j(x)$ is a localization window, $w_j \in \mathbb{R}^d,~\beta_j \in \mathbb{R}$, and $\sigma$ is a nonlinear activation function (here we use $\tanh$).

For homogeneous Dirichlet boundary conditions, $b_\Omega(x)$ is chosen such that $b_\Omega(x)=0$ for $x\in\partial\Omega$. For instance, on the unit hypercube $\Omega=(0,1)^d$, one may take
\begin{equation}
	b_\Omega (x) = \prod_{l=1}^d x_l (1-x_l)
\end{equation}
For nonhomogeneous Dirichlet or mixed boundary conditions, the prescribed boundary data are imposed in the finite-volume system and enter the reduced dynamics through the affine term in Eq.~\eqref{eq:FV_reference_system}. The random-feature expansion is then used for the interior pressure unknowns.

For an internal interface \(\Gamma\), let \(\rho_\Gamma(\mathbf{x})\) denote the signed distance to the interface. We use \(\mathcal N_\Gamma\) to denote a narrow interface neighborhood whose characteristic width is controlled by \(\ell_\Gamma\). A typical smooth localization window associated with this neighborhood is
\begin{equation}
	\omega_\Gamma (x) = \exp \left( -\frac{\rho_\Gamma (x)^2}{\ell_\Gamma^2} \right)
\end{equation}
where $\ell_\Gamma > 0$ controls the thickness of the interface neighborhood. In three dimensions, the same definition applies, with $\Gamma$ interpreted as a surface and $\rho_\Gamma (x)$ as the signed distance to that surface.

The algebraic compression and reduced-evolution steps are dimension-independent. The material-adapted feature dictionary can be instantiated in different ways according to the available geometric information. For the central square benchmark, the known interface geometry is used to construct side-localized features. For the two high-contrast two-dimensional configurations, material masks extracted from the permeability field are used for region-wise feature duplication and enrichment. The three-dimensional cube test uses an analogous region-wise construction. The subsequent compression and reduced-evolution procedures are independent of these case-specific feature-allocation choices.

\subsection{Pressure-space Compression and Whitening}
The full random-feature matrix $\Phi_0$ is generally redundant and ill-conditioned. To obtain a stable reduced coordinate system, we first compress the random feature space in the discrete pressure norm.

Let
\begin{equation}
	A_{\mathrm{int}} = W_h^{1/2} \Phi_0, \notag
\end{equation}
where $W_h$ is a quadrature or mass-weighting matrix associated with the finite-volume degrees of freedom. Compute the thin SVD
\begin{equation}
	A_{\mathrm{int}} = U_{\mathrm{int}} \Sigma_{\mathrm{int}} V_{\mathrm{int}}^\top \notag
\end{equation}

Retaining the dominant right singular vectors gives $T_r = V_{\mathrm{int}, r}$, and the compressed random feature matrix is $\Phi_r = \Phi_0 T_r$. The pressure approximation can then be written as
\begin{equation}
	\mathbf{p}_M(t) = \Phi_r y(t), \quad y(t) \in \mathbb{R}^r
\end{equation}

To construct an orthonormal pressure basis, we further compute
\begin{equation}
	\Phi_r = U_r \Sigma_r V_r^\top \notag
\end{equation}

After truncating small singular values, define
\begin{equation}
	Q = U_{r,k} \in \mathbb{R}^{N_h \times k}, \notag
\end{equation}
then we have $Q^\top Q = I_k$. The final reduced representation is
\begin{equation}
	\mathbf{p}_M(t) = Q a(t), \quad a(t) \in \mathbb{R}^k
\end{equation}

The connection to the original random feature coefficient vector is
\begin{equation}
	z(t) = T_r V_{r,k} \Sigma_{r,k}^{-1} a(t)
\end{equation}

Thus, time evolution is performed in the orthogonal pressure coordinate $a(t)$, while the original random feature representation remains available for continuous evaluation and visualization.

The initial reduced coordinate is obtained by a ridge-regularized pressure projection
\begin{equation}
	a^0 = \arg \min_{a \in \mathbb{R}^k} \left\{ \left\| Q a - \mathbf{p}_h^0 \right\|_2^2 + \eta_0 \left\| a \right\|_2^2 \right\}
\end{equation}

Since $Q^\top Q = I_k$, this gives
\begin{equation}
	a^0 = (1 + \eta_0)^{-1} Q^\top \mathbf{p}_h^0
\end{equation}
When $\eta_0 = 0$, this reduces to the standard orthogonal projection.

\subsection{Reduced Evolution Operator and Time Integration}

The reduced coordinate $a(t)$ is evolved by a low-dimensional linear system
\begin{equation}
	\frac{d a(t)}{dt} = B_{\eta_B} a(t).
\end{equation}

The reduced operator is identified by matching the finite-volume dynamics on the reduced pressure space. The unregularized Galerkin operator is
\begin{equation}
	B_0=Q^\top L_hQ\in\mathbb{R}^{k\times k}.
\end{equation}
Since
\begin{equation}
	\frac{d}{dt} \left( Q a(t) \right) = Q B_{\eta_B} a(t),
\end{equation}
while the finite-volume dynamics gives $L_h Q a(t)$, the regularized operator is determined from
\begin{equation}
	B_{\eta_B} = \arg \min_{\widehat{B} \in \mathbb{R}^{k \times k}} \left\{ \left\| Q \widehat{B} - L_h Q \right\|_F^2 + \eta_B \left\| \widehat{B} \right\|_F^2 \right\}.
\end{equation}

In the implementation, this ridge regression is solved through the normal equation
\begin{equation}
	B_{\eta_B} = \left(Q^\top Q + \eta_B I_k\right)^{-1} Q^\top L_h Q.
\end{equation}

Because $Q^\top Q = I_k$, this reduces to $B_{\eta_B}=(1+\eta_B)^{-1}B_0$. If $\eta_B=0$, then $B_{\eta_B}=B_0$, which is the Galerkin projection of the finite-volume operator onto the orthogonal pressure basis.

To suppress nonphysical Euclidean energy growth caused by finite-precision or truncation effects, we use an optional dissipativity correction. Define the Euclidean logarithmic norm
\begin{equation}
	\mu_2(B_{\eta_B})=
	\lambda_{\max}\!\left(\frac{B_{\eta_B}+B_{\eta_B}^\top}{2}\right)
\end{equation}
and, for a prescribed margin $\delta\geq0$, set
\begin{equation}
	\alpha_{\mathrm{diss}}=
	\max\{0,\mu_2(B_{\eta_B})+\delta\},
	\qquad
	\widetilde B=B_{\eta_B}-\alpha_{\mathrm{diss}}I_k.
\end{equation}
Then $(\widetilde B+\widetilde B^\top)/2\preceq-\delta I_k$. Unlike an eigenvalue-based shift, this construction controls nonnormal transient growth in the Euclidean norm.

Then the reduced homogeneous system
\begin{equation}
	\frac{d a(t)}{dt} = \widetilde B a(t),\quad a(0) = a^0
\end{equation}
is advanced by the Crank--Nicolson scheme:
\begin{equation}
	\left( I_k - \frac{\Delta t}{2} \widetilde B \right) a^{n+1} = \left( I_k + \frac{\Delta t}{2} \widetilde B \right) a^n
\end{equation}

When the finite-volume system contains an affine boundary term in Eq.~\eqref{eq:FV_reference_system}, the reduced coefficient equation becomes
\begin{equation}
	\frac{d a(t)}{dt} = \widetilde B a(t)+Q^\top \mathbf{b}_h(t).
\end{equation}
The corresponding Crank--Nicolson update is
\begin{equation}
	\left( I_k - \frac{\Delta t}{2} \widetilde B \right) a^{n+1}
	=
	\left( I_k + \frac{\Delta t}{2} \widetilde B \right) a^n
	+\frac{\Delta t}{2}Q^\top\left[\mathbf{b}_h(t_{n+1})+\mathbf{b}_h(t_n)\right].
\end{equation}
For time-independent boundary data, the projected forcing term is constant in time.

The pressure approximation is then reconstructed by
\begin{equation}
	\mathbf{p}_M^{n+1} = Q a^{n+1}
\end{equation}

Since $Q^\top Q = I_k$, the reduced coefficient norm is directly tied to the pressure norm in the orthogonal coordinate system.

\begin{remark}
	The construction of the reduced operator $\widetilde B$ is independent of the temporal discretization.
	After $\widetilde B$ has been identified, the coefficient equation $a'(t)=\widetilde B a(t)$ may be advanced by explicit Euler, Crank--Nicolson, fourth-order Runge--Kutta, or other standard ODE solvers.
	We use Crank--Nicolson in the reported experiments to obtain a stable second-order update for the linear reduced dynamics.
\end{remark}

\begin{algorithm}[H]
	\caption{Offline construction and online evolution of the proposed reduced model}
	\label{alg:offline_online_workflow}
	\begin{algorithmic}[1]
		\Require FV matrices $M_h,K_h$, boundary contribution $\mathbf f_h(t)$, initial pressure $\mathbf p_h^0$, permeability field $A$, random-feature configuration, tolerances, $\eta_0,\eta_B,\delta$, time step $\Delta t$, and number of steps $N_t$
		\Ensure Reduced coordinates $\{a^n\}_{n=0}^{N_t}$ and requested pressure snapshots
		\State \textbf{Offline stage}
		\State Form $L_h=M_h^{-1}K_h$ and $\mathbf b_h(t)=M_h^{-1}\mathbf f_h(t)$.
		\State Construct the material-adapted masks and evaluate $\Phi_0$ on the FV degrees of freedom.
		\State Compute $A_{\mathrm{int}}=W_h^{1/2}\Phi_0=U_{\mathrm{int}}\Sigma_{\mathrm{int}}V_{\mathrm{int}}^\top$.
		\State Retain $T_r=V_{\mathrm{int},r}$ and form $\Phi_r=\Phi_0T_r$.
		\State Compute $\Phi_r=U_r\Sigma_rV_r^\top$ and retain $Q=U_{r,k}$.
		\State Set $a^0=(1+\eta_0)^{-1}Q^\top\mathbf p_h^0$.
		\State Form $B_0=Q^\top L_hQ$ and $B_{\eta_B}=(1+\eta_B)^{-1}B_0$.
		\State Compute $\mu_2(B_{\eta_B})=\lambda_{\max}((B_{\eta_B}+B_{\eta_B}^\top)/2)$, set $\alpha_{\mathrm{diss}}=\max\{0,\mu_2(B_{\eta_B})+\delta\}$, and form $\widetilde B=B_{\eta_B}-\alpha_{\mathrm{diss}}I_k$.
		\State Form $G_-=I_k-\frac{\Delta t}{2}\widetilde B$ and $G_+=I_k+\frac{\Delta t}{2}\widetilde B$, and factorize $G_-$ once.
		\State \textbf{Online stage}
		\For{$n=0,\ldots,N_t-1$}
		\State Set $g^n=Q^\top\mathbf b_h(t_n)$ and $g^{n+1}=Q^\top\mathbf b_h(t_{n+1})$.
		\State Solve $G_-a^{n+1}=G_+a^n+\frac{\Delta t}{2}(g^{n+1}+g^n)$.
		\State Reconstruct $\mathbf p_M^{n+1}=Qa^{n+1}$ when a snapshot is requested.
		\EndFor
	\end{algorithmic}
\end{algorithm}

\subsection{Offline--Online Decomposition and Complexity}
\label{sec:complexity}

Let $M_0$ be the number of original random features, $r$ the dimension after the first compression, $k$ the retained orthogonal pressure dimension, $N_t$ the number of time steps, and $N_s$ the number of reconstructed snapshots. Assuming $N_h\geq M_0\geq r\geq k$, the leading costs of the implementation are summarized in Table~\ref{tab:computational_complexity}. The estimates reflect the dense SVD routines used in the present implementation and the sparse structure of the FV operator.

\begin{table}[pos=htbp]
	\centering
	\small
	\caption{Leading computational costs of the offline and online stages.}
	\label{tab:computational_complexity}
	\begin{tabular}{p{0.26\linewidth}p{0.27\linewidth}p{0.34\linewidth}}
		\toprule
		Stage & Arithmetic cost & Principal storage or reuse \\
		\midrule
		Random-feature evaluation &
		$O(N_hM_0)$ &
		$O(N_hM_0)$ for the feature matrix \\
		First thin SVD &
		$O(N_hM_0^2)$ &
		Offline; produces the rank-$r$ compression \\
		Second thin SVD &
		$O(N_hr^2)$ &
		Offline; produces the rank-$k$ pressure basis \\
		Reduced-operator construction &
		$O(\operatorname{nnz}(L_h)k+N_hk^2)$ &
		Offline; $O(N_hk+k^2)$ working storage \\
		Crank--Nicolson factorization &
		$O(k^3)$ &
		Performed once for fixed $\Delta t$ and $\widetilde B$ \\
		Reduced time integration &
		$O(N_tk^2)$ &
		Online; two vectors and the factored $k\times k$ matrix \\
		Pressure reconstruction &
		$O(N_sN_hk)$ &
		Only at requested output times \\
		\bottomrule
	\end{tabular}
\end{table}

For repeated transient queries with the same permeability field, finite-volume operator, and time step, the pressure basis, reduced operator, and Crank--Nicolson factorization can be reused. The per-step reduced solve depends on $k$; projection of an affine forcing term and pressure reconstruction additionally involve the full spatial dimension. The three-dimensional timing results in Section~\ref{sec:scalability_3d} separate these stages: basis construction and operator projection dominate the measured workflow, while reduced time marching is comparatively inexpensive.

\subsection{Reconstruction and Diagnosis}

At time $t_n$, the pressure approximation on the finite-volume degrees of freedom is
\begin{equation}
	\mathbf{p}_M^n = Q a^n \notag
\end{equation}

For off-grid evaluation, one recovers
\begin{equation}
	z^n = T_r V_{r,k} \Sigma_{r,k}^{-1} a^n, \notag
\end{equation}
and evaluates
\begin{equation}
	p_M(x,t_n) = \sum_{j=1}^{M_0} z_j^n \phi_j(x)
\end{equation}

When a finite-volume reference solution is available, we evaluate the pressure reconstruction error at the final time by the relative discrete $L^2$ and $L^\infty$ norms,
\begin{equation}
	E_p^{L^2}
	=
	\frac{\left\|\mathbf{p}_h(T)-\mathbf{p}_M(T)\right\|_2}
	{\left\|\mathbf{p}_h(T)\right\|_2},
	\qquad
	E_p^{L^\infty}
	=
	\frac{\left\|\mathbf{p}_h(T)-\mathbf{p}_M(T)\right\|_\infty}
	{\left\|\mathbf{p}_h(T)\right\|_\infty},
\end{equation}
where $\mathbf{p}_h(T)$ denotes the finite-volume reference pressure and $\mathbf{p}_M(T)$ denotes the reduced-model pressure reconstructed on the same finite-volume grid.

For the Crank--Nicolson update of the reduced dynamics, we define the one-step amplification matrix
\begin{equation}
	R_{\Delta t}
	=
	\left(I_k-\frac{\Delta t}{2}\widetilde B\right)^{-1}
	\left(I_k+\frac{\Delta t}{2}\widetilde B\right),
\end{equation}
and monitor its spectral radius
\begin{equation}
	\rho_{\Delta t}=\rho(R_{\Delta t})
\end{equation}
as an eigenvalue-based stability diagnostic, where $\rho(\cdot)$ denotes the spectral radius. Under the dissipativity condition established in Section~\ref{sec:theory}, the propagation matrix satisfies $\|R_{\Delta t}\|_2\leq1$, giving a contractive update in the Euclidean norm.

Table~\ref{tab:error_metrics_definition} summarizes the reported error metrics and stability diagnostic used in the numerical experiments.

\begin{table}[pos=t]
	\centering
	\small
	\setlength{\tabcolsep}{4pt}
	\caption{Pressure-error metrics and the one-step stability diagnostic.}
	\label{tab:error_metrics_definition}
	\begin{tabular}{lll}
		\toprule
		Symbol & Definition & Purpose \\
		\midrule
		$E_p^{L^2}$ &
		$\displaystyle
		\frac{\|\mathbf{p}_h(T)-\mathbf{p}_M(T)\|_2}
		{\|\mathbf{p}_h(T)\|_2}$ &
		Final-time relative pressure error \\[2ex]
		$E_p^{L^\infty}$ &
		$\displaystyle
		\frac{\|\mathbf{p}_h(T)-\mathbf{p}_M(T)\|_\infty}
		{\|\mathbf{p}_h(T)\|_\infty}$ &
		Final-time maximum relative pressure error \\[2ex]
		$\rho_{\Delta t}$ &
		$\displaystyle
		\rho\!\left[
		\left(I_k-\frac{\Delta t}{2}\widetilde B\right)^{-1}
		\left(I_k+\frac{\Delta t}{2}\widetilde B\right)
		\right]$ &
		One-step stability diagnostic \\
		\bottomrule
	\end{tabular}
\end{table}

These quantities assess final-time pressure reconstruction accuracy and time-integration stability. Problem-specific implementation parameters, feature numbers, truncation ranks, time-step sizes, and benchmark errors are reported separately in the numerical experiments.

\section{Theoretical Analysis}\label{sec:theory}
This section establishes the stability and error properties of the proposed reduced model. The analysis distinguishes four levels of approximation:
\begin{enumerate}
	\item the continuous PDE solution $p(t)$;
	\item the finite-volume solution $\mathbf{p}_h(t)$;
	\item the ideal projected solution $\mathbf{p}_Q(t)$;
	\item the stabilized continuous solution $\mathbf{p}_M(t)$ and its fully discrete counterpart $\mathbf{p}_M^n$.
\end{enumerate}
The total error comprises finite-volume discretization, random-feature approximation, dynamical closure, regularization and stabilization, and Crank--Nicolson time integration. The analysis is formulated on $\Omega\subset\mathbb{R}^d$, with $d\in\{2,3\}$, for homogeneous semi-discrete dynamics. This setting includes homogeneous boundary data and time-independent boundary liftings, as summarized in Remark~\ref{rem:boundary_lifting}.

\subsection{Discrete Setting}
Consider the governing equation~\eqref{eq:PDE} subject to appropriate initial and boundary conditions. The domain may be decomposed into several subdomains
\begin{equation}
	\bar{\Omega} = \bigcup_{m=1}^{M_\Omega} \bar{\Omega}_m \notag
\end{equation}
across whose interfaces the permeability tensor $A(x)$ may be discontinuous.

To match the discretizations used in the numerical experiments, the analysis is restricted to the following coefficient and mesh setting.

\begin{assumption}\label{assumpt:coefficient}
	The storage coefficient is spatially constant, $S(x)\equiv S_0>0$, and the permeability is a positive, piecewise-constant scalar field satisfying
	\begin{equation}
		0<A_{\min}\leq A(x)\leq A_{\max}<\infty
	\end{equation}
	for almost every $x\in\Omega$. The finite-volume mesh is a uniform Cartesian grid, and coefficient jumps are represented cellwise. Diffusive face fluxes are evaluated by a two-point finite-volume formula with harmonic averaging.
\end{assumption}

Let a conservative finite-volume discretization produce the semi-discrete system
\begin{equation}
	M_h \frac{d \mathbf{p}_h(t)}{dt} = K_h \mathbf{p}_h(t), \quad \mathbf{p}_h(t) \in \mathbb{R}^{N_h}
\end{equation}
with
\begin{equation}
	L_h = M_h^{-1} K_h, \notag
\end{equation}
the homogeneous finite-volume reference dynamics becomes $\dot{\mathbf p}_h=L_h\mathbf p_h$. The FVM provides a standard conservative framework for elliptic and parabolic diffusion problems.

In accordance with the Euclidean SVD and orthogonal pressure coordinates used in the implementation, the stability analysis is conducted in the discrete Euclidean inner product,
\begin{equation}
	\left( \mathbf{u}, \mathbf{v} \right)_2 := \mathbf{u}^\top \mathbf{v}, \qquad
	\left\| \mathbf{v} \right\|_2 := \left( \mathbf{v}^\top \mathbf{v} \right)^{1/2}. \notag
\end{equation}
For spatial error estimates, we also use the finite-volume discrete $L^2$ norm
\begin{equation}
	\|\mathbf v\|_h^2:=\sum_{C_i\in\mathcal T_h}|C_i|v_i^2.
\end{equation}
On a uniform grid with cell volume $\omega_h$, $\|\mathbf v\|_h=\omega_h^{1/2}\|\mathbf v\|_2$. Hence, multiplication by this scalar factor does not alter orthogonality, the projector $P_Q$, or relative $L^2$ errors.

The stability analysis requires the following discrete structural condition.

\begin{assumption}[Finite-volume matrix structure]\label{assump:FV_inner_product}
	The finite-volume matrices satisfy
	\begin{equation}
		M_h=m_h I_{N_h},\qquad K_h=K_h^\top\preceq0,\qquad m_h>0.
	\end{equation}
	Consequently, the symmetric part of the finite-volume generator is nonpositive:
	\begin{equation}
		L_h+L_h^\top=\frac{2}{m_h}K_h\preceq0,
	\end{equation}
	or equivalently,
	\begin{equation}
		\mathrm{Re}\!\left( \mathbf{v}^* L_h \mathbf{v} \right) \leq 0, \quad \forall \mathbf{v} \in \mathbb{C}^{N_h}.
	\end{equation}
\end{assumption}

The material-aware random feature dictionary is written as $\Phi_0$. Its feature groups are assigned by the permeability-dependent masks introduced in Section~\ref{sec:method}, so that regions with different coefficient values receive distinct approximation capacity.

After compression and whitening, let $Q \in \mathbb{R}^{N_h \times k}$ satisfy $Q^\top Q = I_k$. The associated DD-RF pressure space is $\mathcal{V}_Q = \mathrm{Range} (Q)$.

\begin{definition}[Orthogonal DD-RF projection]
	The projection from $\mathbb{R}^{N_h}$ onto $\mathcal{V}_Q$ is defined by
	\begin{equation}
		P_Q = Q Q^\top.
	\end{equation}
	The following result establishes its basic property.
\end{definition}

\begin{lemma}[Orthogonal best approximation]\label{lemma:orthogonal_projection}
	The operator $P_Q$ is the Euclidean orthogonal projection onto $\mathcal{V}_Q$. In particular,
	\begin{equation}
		\begin{aligned}
			P_Q^2 &= P_Q \\
			\left( P_Q \mathbf{u}, \mathbf{v} \right)_2 &= \left( \mathbf{u}, P_Q \mathbf{v} \right)_2
		\end{aligned} \notag
	\end{equation}
	and
	\begin{equation}
		\left\| \mathbf{v} - P_Q \mathbf{v} \right\|_2 = \inf_{\mathbf{w} \in \mathcal{V}_Q} \left\| \mathbf{v} - \mathbf{w} \right\|_2. \notag
	\end{equation}
\end{lemma}

The proof is recorded in Appendix~\ref{app:auxiliary_proofs}.

The two compression stages permit a deterministic separation of dictionary and truncation errors. On the uniform grid, $W_h=\omega_hI_{N_h}$. Let $T_r$ contain the first $r$ right singular vectors of $A_{\mathrm{int}}=W_h^{1/2}\Phi_0$, let $\Phi_r=\Phi_0T_r$, and let $Q$ contain the first $k$ left singular vectors of $\Phi_r$, as in Section~\ref{sec:method}. We set $\sigma_{r+1}(A_{\mathrm{int}})=0$ or $\sigma_{k+1}(\Phi_r)=0$ when the corresponding rank is exhausted.

\begin{proposition}[DD-RF compression bound]\label{prop:DDRF_compression}
	On the uniform grid of Assumption~\ref{assumpt:coefficient}, for every $\mathbf v\in\mathbb R^{N_h}$ and $z\in\mathbb R^{M_0}$,
	\begin{equation}
		\begin{aligned}
		\|\mathbf v-P_Q\mathbf v\|_h
		\leq {}& \|\mathbf v-\Phi_0z\|_h
		+\sigma_{r+1}(A_{\mathrm{int}})\|z\|_2 \\
		&+\omega_h^{1/2}\sigma_{k+1}(\Phi_r)\|T_r^\top z\|_2.
		\end{aligned}
	\end{equation}
	Consequently, the projection error is bounded by
	\begin{equation}
		\|\mathbf v-P_Q\mathbf v\|_h\leq \mathcal E_{\mathrm{DDRF},h}(\mathbf v),
	\end{equation}
	where
	\begin{equation}
		\mathcal E_{\mathrm{DDRF},h}(\mathbf v):=
		\inf_{z\in\mathbb R^{M_0}}
		\left\{
		\|\mathbf v-\Phi_0z\|_h
		+\sigma_{r+1}(A_{\mathrm{int}})\|z\|_2
		+\omega_h^{1/2}\sigma_{k+1}(\Phi_r)\|T_r^\top z\|_2
		\right\}.
	\end{equation}
\end{proposition}

\begin{proof}
	Because the $h$-norm is a scalar multiple of the Euclidean norm on the uniform grid, $P_Q$ is also the $h$-orthogonal projector. For any $z$,
	\begin{equation}
		\|\mathbf v-P_Q\mathbf v\|_h
		\leq \|\mathbf v-\Phi_0z\|_h+\|(I-P_Q)\Phi_0z\|_h.
	\end{equation}
	Decompose $z=T_rT_r^\top z+(I-T_rT_r^\top)z$. The discarded part of the first SVD satisfies
	\begin{equation}
		\|\Phi_0(I-T_rT_r^\top)z\|_h
		\leq \sigma_{r+1}(A_{\mathrm{int}})\|z\|_2,
	\end{equation}
	whereas the rank-$k$ SVD of $\Phi_r$ gives
	\begin{equation}
		\|(I-P_Q)\Phi_rT_r^\top z\|_h
		\leq \omega_h^{1/2}\sigma_{k+1}(\Phi_r)\|T_r^\top z\|_2.
	\end{equation}
	The first claim follows by the triangle inequality; taking the infimum over $z$ proves the second.
\end{proof}

The implementation may use a ridge-regularized projection of the initial state. Let $\mathbf{p}_h^0 = \mathbf{p}_0$ and define
\begin{equation}
	a_{\eta_0}^0 := \arg \min_{a \in \mathbb{R}^k} \left\{ \left\| Q a - \mathbf{p}_h^0 \right\|_2^2 + \eta_0 \left\| a \right\|_2^2 \right\}, \quad \eta_0 \geq 0.
\end{equation}

\begin{lemma}[Regularized initial projection]\label{lemma:regularized_initial_projection}
	The regularized initial coordinate is
	\begin{equation}
		a_{\eta_0}^0 = \frac{1}{1 + \eta_0} Q^\top \mathbf{p}_h^0
	\end{equation}
	Moreover, 
	\begin{equation}
		\left\| P_Q \mathbf{p}_h^0 - Q a_{\eta_0}^0 \right\|_2 = \frac{\eta_0}{1 + \eta_0} \left\| P_Q \mathbf{p}_h^0 \right\|_2.
	\end{equation}
\end{lemma}

The proof is given in Appendix~\ref{app:auxiliary_proofs}.

\subsection{Dissipative Reduced Dynamics}

The ideal reduced operator is identified by the Euclidean Galerkin projection
\begin{equation}
	B_0 = Q^\top L_h Q.
\end{equation}
In the actual computation, a ridge-regularized operator may be obtained from
\begin{equation}
	B_{\eta_B} = \arg \min_{B \in \mathbb{R}^{k \times k}} \left\{ \left\| Q B - L_h Q \right\|_F^2 + \eta_B \left\| B \right\|_F^2 \right\},
\end{equation}
where $\|\cdot\|_F$ denotes the Frobenius norm.

\begin{lemma}[Ridge--Galerkin equivalence]\label{lemma:operator_regression}
	The unique minimizer of the operator-regression problem is
	\begin{equation}
		B_{\eta_B} = \frac{1}{1 + \eta_B} Q^\top L_h Q = \frac{1}{1 + \eta_B} B_0.
	\end{equation}
\end{lemma}

The proof is given in Appendix~\ref{app:auxiliary_proofs}.

For a nonnormal matrix, the symmetric part governs instantaneous Euclidean energy growth. The optional dissipativity correction uses the Euclidean logarithmic norm
\begin{equation}
	\mu_2(C):=\lambda_{\max}\!\left(\frac{C+C^\top}{2}\right).
\end{equation}
For a prescribed margin $\delta\geq0$, define
\begin{equation}\label{eq:dissipativity_shift}
	\alpha_{\mathrm{diss}}:=\max\{0,\mu_2(B_{\eta_B})+\delta\},
	\qquad
	\widetilde B:=B_{\eta_B}-\alpha_{\mathrm{diss}}I_k.
\end{equation}
This definition corrects the symmetric part directly, including in the nonnormal case.

\begin{theorem}[Inherited dissipativity]\label{theo:reduced_stability}
	Under Assumption~\ref{assump:FV_inner_product}, we have
	\begin{equation}
		B_0+B_0^\top\preceq0,\qquad
		B_{\eta_B}+B_{\eta_B}^\top\preceq0,\qquad
		\widetilde B+\widetilde B^\top\preceq-2\delta I_k.
	\end{equation}
	If $a(t)$ solves
	\begin{equation}
		\dot{a}(t) = \widetilde B a(t),
	\end{equation}
	then
	\begin{equation}
		\|a(t)\|_2\leq e^{-\delta t}\|a(0)\|_2,
	\end{equation}
	holds, and the reconstructed pressure satisfies
	\begin{equation}
		\|Qa(t)\|_2\leq e^{-\delta t}\|Qa(0)\|_2.
	\end{equation}
\end{theorem}

\begin{proof}
	By definition, $B_0 = Q^\top L_h Q$ and $B_0^\top = Q^\top L_h^\top Q$. Therefore,
	\begin{equation}
		B_0 + B_0^\top = Q^\top \left( L_h + L_h^\top \right) Q.
	\end{equation}
	For every $a \in \mathbb{R}^k$,
	\begin{equation}
		a^\top \left( B_0 + B_0^\top \right) a = \left( Q a \right)^\top \left( L_h + L_h^\top \right) \left( Qa \right).
	\end{equation}
	
	Assumption~\ref{assump:FV_inner_product} gives $L_h + L_h^\top \preceq 0$. Hence, $a^\top \left( B_0 + B_0^\top \right) a \leq 0$, which proves $B_0 + B_0^\top \preceq 0$.
	
	Since $B_{\eta_B} = \frac{1}{1 + \eta_B} B_0$ and $\left( 1 + \eta_B \right)^{-1} > 0$,
	\begin{equation}
		B_{\eta_B} + B_{\eta_B}^\top = \frac{1}{1 + \eta_B} \left( B_0 + B_0^\top \right) \preceq 0.
	\end{equation}
	By Eq.~\eqref{eq:dissipativity_shift}, either $\alpha_{\mathrm{diss}}=0$ and $\mu_2(B_{\eta_B})\leq-\delta$, or $\alpha_{\mathrm{diss}}=\mu_2(B_{\eta_B})+\delta$. Thus
	\begin{equation}
		\mu_2(\widetilde B)=\mu_2(B_{\eta_B})-\alpha_{\mathrm{diss}}\leq-\delta,
	\end{equation}
	which is equivalent to $\widetilde B+\widetilde B^\top\preceq-2\delta I_k$.
	
	Let $a(t)$ be the solution of the system 
	\begin{equation}
		\dot{a} = \widetilde B a.
	\end{equation}
	Differentiating the squared Euclidean norm gives
	\begin{equation}
		\frac{d}{dt} \left\| a(t) \right\|_2^2 = 2a(t)^\top \widetilde B a(t).
	\end{equation}
	Since the skew-symmetric part does not contribute to a real quadratic form, $2a^\top \widetilde B a=a^\top(\widetilde B+\widetilde B^\top)a$. Therefore,
	\begin{equation}
		\frac{d}{dt}\|a(t)\|_2^2\leq-2\delta\|a(t)\|_2^2. \notag
	\end{equation}
	Multiplication by $e^{2\delta t}$ and integration from $0$ to $t$ give
	\begin{equation}
		\|a(t)\|_2\leq e^{-\delta t}\|a(0)\|_2.
	\end{equation}
	
	By the above derivation, we know that
	\begin{equation}
		\left\| Qa(t) \right\|_2^2 = a(t)^\top Q^\top Q a(t) = a(t)^\top a(t) = \left\| a(t) \right\|_2^2.
	\end{equation}
	The same identity holds at $t=0$, which yields
	\begin{equation}
		\|Qa(t)\|_2\leq e^{-\delta t}\|Qa(0)\|_2.
	\end{equation}
	
	This completes the proof and yields the following corollary.
\end{proof}

\begin{corollary}[Reduced semigroup bounds]\label{coro:semigroup_contractivity}
	For all $t\geq0$,
	\begin{equation}
		\|e^{tB_0}\|_2\leq1,\qquad
		\|e^{tB_{\eta_B}}\|_2\leq1,\qquad
		\|e^{t\widetilde B}\|_2\leq e^{-\delta t}.
	\end{equation}
	In particular, every eigenvalue $\lambda\in\sigma(\widetilde B)$ satisfies $\operatorname{Re}\lambda\leq-\delta$.
\end{corollary}

The proof follows directly from the energy estimate and is included in Appendix~\ref{app:auxiliary_proofs}.

\begin{remark}[Time-independent nonhomogeneous boundary data]\label{rem:boundary_lifting}
	For time-independent nonhomogeneous boundary conditions, the finite-volume system has the affine form
	\begin{equation}
		\dot{\mathbf p}_h=L_h\mathbf p_h+\mathbf b_h.
	\end{equation}
	If a discrete steady lifting $\bar{\mathbf p}_h$ satisfies
	\begin{equation}
		L_h\bar{\mathbf p}_h+\mathbf b_h=0,
	\end{equation}
	then $\mathbf u_h(t):=\mathbf p_h(t)-\bar{\mathbf p}_h$ obeys $\dot{\mathbf u}_h=L_h\mathbf u_h$. Likewise, for $\dot a=\widetilde B a+Q^\top\mathbf b_h$, any reduced equilibrium $\bar a$ satisfying $\widetilde B\bar a+Q^\top\mathbf b_h=0$ yields $c(t):=a(t)-\bar a$ with $\dot c=\widetilde Bc$. The resulting lifting term is incorporated explicitly in Corollary~\ref{coro:affine_total_error}.
\end{remark}

\subsection{DD-RF Approximation Error}

The ideal projected solution in $\mathcal{V}_Q$ separates pressure-space approximation from subsequent algebraic and temporal perturbations. It provides the analytical comparison trajectory for the implemented reduced solution.

Let the ideal projected DD-RF solution satisfy
\begin{equation}
	\frac{d \mathbf{p}_Q(t)}{dt} = P_Q L_h \mathbf{p}_Q(t), \quad \mathbf{p}_Q(0) = P_Q \mathbf{p}_h^0 \notag.
\end{equation}
Because $\mathbf{p}_Q(t) \in \mathcal{V}_Q$, it can be written as
\begin{equation}
	\mathbf{p}_Q(t) = Q a_0(t),
\end{equation}
where $\dot{a}_0(t) = B_0 a_0(t)$ and $a_0(0) = Q^\top \mathbf{p}_h^0$.

\begin{definition}[Projection--closure quantities]
	The instantaneous DD-RF projection error is defined by
	\begin{equation}
		\varepsilon_Q(t) := \left\| \left( I - P_Q \right) \mathbf{p}_h(t) \right\|_2.
	\end{equation}
	The unresolved-to-resolved closure residual is
	\begin{equation}
		\chi_Q(t) := \left\| P_Q L_h \left( I - P_Q \right) \mathbf{p}_h(t) \right\|_2.
	\end{equation}
	The first quantity measures the state approximation, while the second measures the influence of the unresolved pressure component on the resolved dynamics.
\end{definition}

Proposition~\ref{prop:DDRF_compression} makes the trial-space contribution explicit. In the discrete pressure norm,
\begin{equation}\label{eq:DDRF_projection_bound}
	\varepsilon_{Q,h}(t):=\omega_h^{1/2}\varepsilon_Q(t)
	\leq \mathcal E_{\mathrm{DDRF},h}(\mathbf p_h(t)).
\end{equation}
The three terms in $\mathcal E_{\mathrm{DDRF},h}$ represent dictionary approximation, first-stage compression, and final rank truncation, respectively.

\begin{remark}
	Proposition 3.5 provides a deterministic error decomposition for a fixed realization of the random-feature dictionary.
\end{remark}

Define $\rho(t) = \left( I - P_Q \right) \mathbf{p}_h(t)$ and $\theta(t) = P_Q \mathbf{p}_h(t) - \mathbf{p}_Q(t)$, then $\mathbf{p}_h(t) - \mathbf{p}_Q(t) = \rho(t) + \theta(t)$.

\begin{lemma}[Evolution equation for the resolved error]\label{lemma:resolved_error_evolution}
	The error component $\theta(t) \in \mathcal{V}_Q$ satisfies
	\begin{equation}
		\dot{\theta}(t) = P_Q L_h  \theta(t) + P_Q L_h \rho(t).
	\end{equation}
\end{lemma}

\begin{proof}
	Since $P_Q$ is time-independent, we have $\dot{\theta}(t) = P_Q \dot{\mathbf{p}}_h(t) - \dot{\mathbf{p}}_Q(t)$. Using $\dot{\mathbf{p}}_h = L_h \mathbf{p}_h$ and $\dot{\mathbf{p}}_Q = P_Q L_h \mathbf{p}_Q$ gives
	\begin{equation}
		\dot{\theta} = P_Q L_h \mathbf{p}_h - P_Q L_h \mathbf{p}_Q = P_Q L_h \left( \mathbf{p}_h - \mathbf{p}_Q \right) = P_Q L_h \rho + P_Q L_h \theta. \notag
	\end{equation}
	with $\theta(0) = P_Q \mathbf{p}_h^0 - \mathbf{p}_Q(0) = 0$.
\end{proof}

\begin{theorem}[Projected evolution error]\label{theo:DDRF_error}
	Under Assumption~\ref{assump:FV_inner_product}, we have the following error estimate for the random feature-based pressure:
	\begin{equation}
		\left\| \mathbf{p}_h(t) - \mathbf{p}_Q(t) \right\|_2 \leq \varepsilon_Q(t) + \int_{0}^t \chi_Q(s) ds.
	\end{equation}
\end{theorem}

\begin{proof}
	Using $\mathbf{p}_h - \mathbf{p}_Q = \rho + \theta$ and taking the triangle inequality, we have
	\begin{equation}
		\left\| \mathbf{p}_h(t) - \mathbf{p}_Q(t) \right\|_2 \leq \left\| \rho(t) \right\|_2 + \left\| \theta(t) \right\|_2.
	\end{equation}
	By definition, $\left\| \rho(t) \right\|_2 = \varepsilon_Q(t)$. It remains to estimate $\theta$. Because $\theta(t) \in \mathcal{V}_Q$, there exists $b(t) \in \mathbb{R}^k$ such that $\theta(t) = Q b(t)$. Multiplying the equation in Lemma~\ref{lemma:resolved_error_evolution} by $Q^\top$ gives
	\begin{equation}
		Q^\top \dot{\theta} = Q^\top P_Q L_h \theta + Q^\top P_Q L_h \rho. \notag
	\end{equation}
	Since $Q^\top P_Q = Q^\top Q Q^\top = Q^\top$ and $\theta = Qb$,
	\begin{equation}
		\dot{b} = B_0 b + Q^\top L_h \rho. \notag
	\end{equation}
	The initial condition is $b(0)=0$. By the variation-of-constants formula,
	\begin{equation}
		b(t) = \int_0^t e^{\left( t-s \right) B_0} Q^\top L_h \rho(s) ds.
	\end{equation}
	
	By Corollary~\ref{coro:semigroup_contractivity}, $\|e^{(t-s)B_0}\|_2\leq1$. Hence,
	\begin{equation}
		\left\| b(t) \right\|_2 \leq \int_0^t \left\| Q^\top L_h \rho(s) \right\|_2 ds.
	\end{equation}
	For every $z \in \mathbb{R}^{N_h}$, the orthonormality of $Q$ gives
	\begin{equation}
		\left\| P_Q z \right\|_2^2 = z^\top Q Q^\top Q Q^\top z = \left\| Q^\top z \right\|_2^2
	\end{equation}
	and
	\begin{equation}
		\left\| Q^\top L_h \rho(s) \right\|_2 = \left\| P_Q L_h \rho(s) \right\|_2.
	\end{equation}
	Since $\rho(s) = \left( I - P_Q \right) \mathbf{p}_h(s)$, we have
	\begin{equation}
		\left\| Q^\top L_h \rho(s) \right\|_2 = \left\| P_Q L_h \left( I - P_Q \right) \mathbf{p}_h(s) \right\|_2 = \chi_Q(s).
	\end{equation}
	Therefore, $\left\| b(t) \right\|_2 \leq \int_0^t \chi_Q(s) ds$. Using the orthonormality of $Q$,
	\begin{equation}
		\left\| \theta(t) \right\|_2 = \left\| Qb(t) \right\|_2 = \left\| b(t) \right\|_2.
	\end{equation}
	Consequently, we have
	\begin{equation}
		\left\| \theta(t) \right\|_2 \leq \int_0^t \chi_Q(s) ds.
	\end{equation}
	Combining the bounds for $\rho$ and $\theta$, we can obtain
	\begin{equation}
		\left\| \mathbf{p}_h(t) - \mathbf{p}_Q(t) \right\|_2 \leq \varepsilon_Q(t) + \int_0^t \chi_Q(s) ds.
	\end{equation}
\end{proof}

The closure error can be related directly to the projection error, provided that the unresolved-to-resolved coupling is bounded.

\begin{corollary}[Closure-coupling bound]\label{coro:closure_coupling}
	Suppose
	\begin{equation}
		\gamma_Q = \left\| P_Q L_h \left( I - P_Q \right) \right\|_{2 \rightarrow 2} < \infty,
	\end{equation}
	then we have
	\begin{equation}
		\chi_Q(t) \leq \gamma_Q \varepsilon_Q(t)
	\end{equation}
	and
	\begin{equation}
		\left\| \mathbf{p}_h(t) - \mathbf{p}_Q(t) \right\|_2 \leq \varepsilon_Q(t) + \gamma_Q \int_0^t \varepsilon_Q(s) ds.
	\end{equation}
\end{corollary}

The proof is given in Appendix~\ref{app:auxiliary_proofs}.

\begin{remark}
	The commonly reported operator residual
	\begin{equation}
		\mathcal{R}_Q = L_hQ - QB_0 = \left( I - P_Q \right) L_hQ \notag
	\end{equation}
measures the resolved-to-unresolved action of $L_h$, whereas $P_Q L_h\left(I-P_Q\right)$ measures unresolved-to-resolved feedback. For a self-adjoint $L_h$ in the Euclidean inner product, the corresponding coupling operators are adjoints. The residual $\mathcal R_Q$ diagnoses invariance of the pressure space; $\chi_Q$ quantifies trajectory-dependent closure feedback.
\end{remark}

\subsection{Regularized Time Evolution}

We compare the ideal reduced system $\dot{a}_0=B_0a_0$, $a_0(0)=Q^\top\mathbf p_h^0$, with the implemented continuous system $\dot{\widetilde a}=\widetilde B\widetilde a$, $\widetilde a(0)=a_{\eta_0}^0$, where $\widetilde B=(1+\eta_B)^{-1}B_0-\alpha_{\mathrm{diss}}I_k$. The corresponding reconstruction is
\begin{equation}
	\mathbf{p}_M(t) = Q\widetilde a(t).
\end{equation}
Thus, $\mathbf p_Q(t)$ and $\mathbf p_M(t)$ coincide only when $\eta_0=\eta_B=\alpha_{\mathrm{diss}}=0$; the fully discrete output is denoted by $\mathbf p_M^n=Qa^n$.

Let
\begin{equation}
	\delta_B := \left\| \widetilde B - B_0 \right\|_2
\end{equation}
be the operator perturbation. From
\begin{equation}
	\begin{aligned}
		\widetilde B - B_0 &= \frac{1}{1 + \eta_B} B_0 - \alpha_{\mathrm{diss}} I_k - B_0 \\
		&= -\frac{\eta_B}{1 + \eta_B} B_0 - \alpha_{\mathrm{diss}} I_k
	\end{aligned}
\end{equation}
and the triangle inequality, we obtain
\begin{equation}
	\left\| \widetilde B - B_0 \right\|_2 \leq \frac{\eta_B}{1 + \eta_B} \left\| B_0 \right\|_2 + \alpha_{\mathrm{diss}}.
\end{equation}
Then we have
\begin{equation}
	\delta_B \leq \frac{\eta_B}{1 + \eta_B} \left\| B_0 \right\|_2 + \alpha_{\mathrm{diss}}.
\end{equation}

\begin{theorem}[Continuous perturbation error]\label{theo:continuous_errors_regularization_shift}
	For $0 \leq t \leq T$, the following inequalities hold:
	\begin{equation}
		\left\| a_0(t) - \widetilde a(t) \right\|_2 \leq \frac{\eta_0}{1 + \eta_0} \left\| P_Q \mathbf{p}_h^0 \right\|_2 + t \delta_B \left\| a_{\eta_0}^0 \right\|_2,
	\end{equation}
	which implies that
	\begin{equation}
		\left\| a_0(t) - \widetilde a(t) \right\|_2 \leq \frac{\eta_0}{1 + \eta_0} \left\| P_Q \mathbf{p}_h^0 \right\|_2 + t \left( \frac{\eta_B}{1 + \eta_B} \left\| B_0 \right\|_2 + \alpha_{\mathrm{diss}} \right)  \left\| a_{\eta_0}^0 \right\|_2.
	\end{equation}
\end{theorem}

\begin{proof}
	Define $e_B(t):=a_0(t)-\widetilde a(t)$. Then
	\begin{equation}
		\dot{e}_B = B_0 a_0 - \widetilde B \widetilde a = B_0 \left( a_0 - \widetilde a \right) + \left( B_0 - \widetilde B \right) \widetilde a.
	\end{equation}
	with initial error $e_B(0) = a_0(0) - a_{\eta_0}^0$.
	
	By the variation-of-constants formula,
	\begin{equation}
		e_B(t) = e^{tB_0} e_B(0) + \int_0^t e^{(t-s)B_0} \left( B_0 - \widetilde B \right) \widetilde a(s) ds.
	\end{equation}
	
	Theorem~\ref{theo:reduced_stability} implies $\left\| e^{tB_0} \right\|_2 \leq 1$. It also gives
	\begin{equation}
		\left\| \widetilde a(s) \right\|_2 \leq e^{-\delta s} \left\| a_{\eta_0}^0 \right\|_2 \leq \left\| a_{\eta_0}^0 \right\|_2.
	\end{equation}
	Taking norms yields $\left\| e_B(t) \right\|_2 \leq \left\| e_B(0) \right\|_2 + \int_0^t \left\| B_0 - \widetilde B \right\|_2 \left\| \widetilde a(s) \right\|_2 ds$ and hence $\left\| e_B(t) \right\|_2 \leq \left\| e_B(0) \right\|_2 + t \delta_B \left\| a_{\eta_0}^0 \right\|_2$.
	
	It remains to estimate $e_B(0)$. Since $a_0(0) = Q^\top \mathbf{p}_h^0$ and $a_{\eta_0}^0 = \frac{1}{1 + \eta_0} Q^\top \mathbf{p}_h^0$, we have $e_B(0)=\frac{\eta_0}{1+\eta_0} Q^\top \mathbf{p}_h^0$. Moreover, $\left\| Q^\top \mathbf{p}_h^0 \right\|_2 = \left\| P_Q \mathbf{p}_h^0 \right\|_2$. Therefore,
	\begin{equation}
		\left\| e_B(0) \right\|_2
		= \frac{\eta_0}{1+\eta_0} \left\| P_Q \mathbf{p}_h^0 \right\|_2,
	\end{equation}
	which provides that
	\begin{equation}
		\left\| a_0(t) - \widetilde a(t) \right\|_2
		\leq
		\frac{\eta_0}{1+\eta_0} \left\| P_Q \mathbf{p}_h^0 \right\|_2 + t \delta_B \left\| a_{\eta_0}^0 \right\|_2.
	\end{equation}
\end{proof}

The stabilized reduced system is discretized by Crank--Nicolson:
\begin{equation}
	\left( I_k - \frac{\Delta t}{2} \widetilde B \right) a^{n+1} = \left( I_k + \frac{\Delta t}{2} \widetilde B\right) a^n.
\end{equation}
The next result is stronger than a modal eigenvalue argument: it proves contractivity even when $\widetilde B$ is non-normal.

\begin{theorem}[Well-posedness and unconditional contractivity of Crank--Nicolson]\label{theo_well_posedness}
	Suppose $\widetilde B + \widetilde B^\top \preceq 0$. Then, for every $\Delta t >0$, $I_k - \frac{\Delta t}{2} \widetilde B$ is invertible, and the Crank--Nicolson update satisfies $\left\| a^{n+1} \right\|_2 \leq \left\| a^n \right\|_2$. Equivalently,
	\begin{equation}
		\left\| \left( I_k - \frac{\Delta t}{2} \widetilde B \right)^{-1} \left( I_k + \frac{\Delta t}{2} \widetilde B \right) \right\|_2 \leq 1.
	\end{equation}
\end{theorem}

\begin{proof}
	Set $c=\frac{\Delta t}{2}>0$. We first prove that $I_k - c\widetilde B$ is invertible. Suppose $(I_k-c\widetilde B)y=0$, we have $y=c\widetilde By$. Taking the Euclidean inner product with $y$ and taking real parts,
	\begin{equation}
		\left\| y \right\|_2^2 = c \mathrm{Re} \left( \widetilde B y, y \right).
	\end{equation}
	Because $\widetilde B$ is dissipative, $\mathrm{Re}(\widetilde By,y)=\frac{1}{2} y^* \left( \widetilde B + \widetilde B^* \right) y \leq 0$. Therefore, $\left\| y \right\|_2^2 \leq 0$ holds, which implies $y=0$. Thus $I_k - c \widetilde B$ has a trivial kernel and is invertible in finite dimensions.
	
	Writing and rearranging the Crank--Nicolson equation gives
	\begin{equation}
		a^{n+1} - a^n = c \widetilde B \left( a^{n+1} + a^n \right).
	\end{equation}
	Use the polarization identity
	\begin{equation}
		\| a^{n+1} \|_2^2 - \| a^n \|_2^2 = \mathrm{Re} \left( a^{n+1}-a^n,a^{n+1}+a^n\right).
	\end{equation}
	The substitution gives
	\begin{equation}
		\| a^{n+1} \|_2^2 - \| a^n \|_2^2 = c \mathrm{Re} \left( \widetilde B \left( a^{n+1} + a^n \right), a^{n+1}+a^n\right) \leq 0.
	\end{equation}
	Then we have $\| a^{n+1}\|_2^2 \leq \|a^n\|_2^2$ and $\|a^{n+1}\|_2 \leq \| a^n\|_2$.
	Since the inequality holds for arbitrary $a^n$, the induced norm of the Crank--Nicolson propagation matrix is at most one.
\end{proof}

We next prove the second-order temporal error.

\begin{theorem}[Second-order Crank--Nicolson error]\label{theo:second_CN_error}
	Let $\widetilde a(t)$ be the solution of $\dot{\widetilde a}=\widetilde B\widetilde a,~\widetilde a(0)=a_{\eta_0}^0$ and $a^n$ be its Crank--Nicolson approximation. If $\widetilde a \in C^3 \left( [0,T];\mathbb{R}^k\right)$, then
	\begin{equation}
		\left\| \widetilde a(t_n)-a^n \right\|_2 \leq \frac{T}{12}\Delta t^2 \max_{0\leq t \leq T} \left\| \widetilde a^{(3)}(t) \right\|_2.
	\end{equation}
	Since $\widetilde a^{(3)}(t)=\widetilde B^3 \widetilde a(t)$, we also have
	\begin{equation}
		\left\| \widetilde a(t_n)-a^n \right\|_2 \leq \frac{T}{12} \Delta t^2 \max_{0\leq t\leq T} \left\| \widetilde B^3 \widetilde a(t) \right\|_2.
	\end{equation}
\end{theorem}

\begin{proof}
	The exact reduced solution satisfies
	\begin{equation}
		\widetilde a(t_{n+1})-\widetilde a(t_n)=\int_{t_n}^{t_{n+1}} \widetilde a'(s) ds. \notag
	\end{equation}
	Applying the trapezoidal rule to the integral, we have
	\begin{equation}
		\int_{t_n}^{t_{n+1}}\widetilde a'(s) ds = \frac{\Delta t}{2} \left[ \widetilde a'(t_n)+\widetilde a'(t_{n+1}) \right] + d^{n+1},
	\end{equation}
	where the trapezoidal defect satisfies $\left\| d^{n+1} \right\|_2 \leq \frac{\Delta t^3}{12} \max_{t_n \leq t \leq t_{n+1}} \left\| \widetilde a^{(3)}(t)\right\|_2$. Since $\widetilde a'(t)=\widetilde B\widetilde a(t)$, the exact solution satisfies
	\begin{equation}
		\left( I_k - \frac{\Delta t}{2}\widetilde B \right) \widetilde a(t_{n+1})
		=
		\left( I_k + \frac{\Delta t}{2}\widetilde B\right) \widetilde a(t_n)+d^{n+1}.
	\end{equation}
	
	Define $e^n=\widetilde a(t_n)-a^n$. Subtracting the numerical equation from the exact defect equation gives
	\begin{equation}
		\left( I_k - \frac{\Delta t}{2}\widetilde B\right) e^{n+1}
		=
		\left( I_k + \frac{\Delta t}{2}\widetilde B\right) e^n + d^{n+1}.
	\end{equation}
	Hence, $e^{n+1}=R_{\Delta t}(\widetilde B)e^n+\left( I_k - \frac{\Delta t}{2}\widetilde B\right)^{-1} d^{n+1}$, where $R_{\Delta t}(\widetilde B)=
	\left( I_k-\frac{\Delta t}{2}\widetilde B\right)^{-1}\left( I_k+\frac{\Delta t}{2}\widetilde B\right)$.
	
	By Theorem~\ref{theo_well_posedness}, we know that $\left\| R_{\Delta t}(\widetilde B)\right\|_2 \leq 1$. Let $c=\frac{\Delta t}{2}$ and $z=\left( I_k-c\widetilde B\right)y$, we have
	\begin{equation}
		\left\| z \right\|_2^2
		=
		\left\| y-c\widetilde By \right\|_2^2
		=
		\left\| y \right\|_2^2 - 2c\mathrm{Re}(\widetilde By,y) + c^2\left\|\widetilde By\right\|_2^2.
	\end{equation}
	Since $\widetilde B$ is dissipative, $\mathrm{Re}(\widetilde By,y)\leq 0$. So we have $\| z \|_2^2 \geq \| y \|_2^2$ and $\| y \|_2 \leq \| z \|_2$. This result proves $\left\| (I_k-c\widetilde B)^{-1}\right\|_2\leq 1$. Consequently, we have
	\begin{equation}
		\left\| e^{n+1} \right\|_2 \leq \left\|e^n\right\|_2 + \left\| d^{n+1}\right\|_2. \notag
	\end{equation}
	
	Since the exact and numerical initial conditions coincide, $e^0=0$ holds. Iterating the inequality gives
	\begin{equation}
		\left\| e^n \right\|_2 \leq \sum_{j=0}^{n-1} \left\| d^{j+1} \right\|_2.
	\end{equation}
	Using the defect bound, we have
	\begin{equation}
		\left\| e^n \right\|_2 \leq n\frac{\Delta t^3}{12} \max_{0\leq t\leq T} \left\| \widetilde a^{(3)}(t) \right\|_2.
	\end{equation}
	Since $t_n=n\Delta t\leq T$, it follows that $n\Delta t^3=t_n\Delta t^2\leq T\Delta t^2$. Therefore,
	\begin{equation}
		\left\| e^n \right\|_2 \leq \frac{T}{12}\Delta t^2 \max_{0\leq t \leq T} \left\| \widetilde a^{(3)}(t)\right\|_2.
	\end{equation}
	
	For the autonomous linear system $\widetilde a'=\widetilde B\widetilde a$, successive differentiation gives
	\begin{equation}
		\widetilde a''=\widetilde B^2\widetilde a,\quad \widetilde a^{(3)}=\widetilde B^3\widetilde a.
	\end{equation}
	The second estimate follows.
\end{proof}

\subsection{Grid-Level Error Estimate}
To compare the numerical solution with the continuous solution, let $\mathcal{I}_h p(t) \in \mathbb{R}^{N_h}$ denote the vector of exact cell averages,
\begin{equation}
	\left(\mathcal I_h p(t)\right)_i:=\frac{1}{|C_i|}\int_{C_i}p(x,t)\,dx.
\end{equation}

Set $\mathbf u_h(t):=\mathcal I_hp(t)$. Unlike the finite-volume solution, $\mathbf u_h$ does not in general satisfy the semi-discrete equation exactly.

\begin{assumption}[Finite-volume consistency]\label{assump:FV_consistency}
	For some $s_{\mathrm{FV}}>0$ and mesh-independent constants $C_{\mathrm{con}},C_0$, the residual
	\begin{equation}
		\tau_h(t):=\dot{\mathbf u}_h(t)-L_h\mathbf u_h(t)
	\end{equation}
	satisfies
	\begin{equation}
		\|\tau_h(t)\|_h\leq C_{\mathrm{con}}h^{s_{\mathrm{FV}}},
		\qquad 0\leq t\leq T,
	\end{equation}
	and the initial discretization obeys
	\begin{equation}
		\|\mathbf u_h(0)-\mathbf p_h^0\|_h\leq C_0h^{s_{\mathrm{FV}}}.
	\end{equation}
\end{assumption}

\begin{proposition}[Finite-volume evolution error]\label{prop:FV_error}
	Under Assumptions~\ref{assump:FV_inner_product} and~\ref{assump:FV_consistency},
	\begin{equation}
		\|\mathcal I_hp(t)-\mathbf p_h(t)\|_h
		\leq \bigl(C_0+tC_{\mathrm{con}}\bigr)h^{s_{\mathrm{FV}}}
		\leq C_{\mathrm{FV}}(T)h^{s_{\mathrm{FV}}},
	\end{equation}
	where $C_{\mathrm{FV}}(T):=C_0+TC_{\mathrm{con}}$. The order $s_{\mathrm{FV}}$ depends on the flux construction, interface alignment, and piecewise regularity; no universal order is asserted here.
\end{proposition}

\begin{proof}
	With $\mathbf e_h:=\mathbf u_h-\mathbf p_h$, subtraction of the semi-discrete equation gives
	\begin{equation}
		\dot{\mathbf e}_h=L_h\mathbf e_h+\tau_h.
	\end{equation}
	Assumption~\ref{assump:FV_inner_product} and the uniform-grid relation between $\|\cdot\|_h$ and $\|\cdot\|_2$ imply $\|e^{tL_h}\|_{h\to h}\leq1$. Variation of constants therefore yields
	\begin{equation}
		\|\mathbf e_h(t)\|_h
		\leq\|\mathbf e_h(0)\|_h+
		\int_0^t\|\tau_h(s)\|_h\,ds,
	\end{equation}
	and the stated estimate follows from Assumption~\ref{assump:FV_consistency}.
\end{proof}

Write the fully discrete DD RF-EDNN approximation as $\mathbf p_M^n=Qa^n$, and define the pressure-scaled closure quantity
\begin{equation}
	\chi_{Q,h}(t):=\omega_h^{1/2}\chi_Q(t).
\end{equation}

\begin{theorem}[Conditional grid-level estimate]\label{theo:total_error}
	Suppose the assumptions of the preceding results hold and $\widetilde a\in C^3([0,T];\mathbb R^k)$. Then
	\begin{equation}
		\begin{aligned}
			\left\| \mathcal{I}_h p(t_n) - \mathbf{p}_M^n \right\|_h
			\leq {}& C_{\mathrm{FV}}(T) h^{s_{\mathrm{FV}}}
			+ \mathcal E_{\mathrm{DDRF},h}(\mathbf p_h(t_n))
			+ \int_{0}^{t_n} \chi_{Q,h}(s)\,ds \\
			&+ \omega_h^{1/2}\frac{\eta_0}{1 + \eta_0}
			\left\| P_Q \mathbf{p}_h^0 \right\|_2 \\
			&+ \omega_h^{1/2}t_n
			\left( \frac{\eta_B}{1 + \eta_B}\left\| B_0 \right\|_2 + \alpha_{\mathrm{diss}} \right)
			\left\| a_{\eta_0}^0 \right\|_2 \\
			&+ \omega_h^{1/2}\frac{T}{12}\Delta t^2
			\max_{0 \leq t \leq T}
			\left\| \widetilde B^3 \widetilde a(t) \right\|_2.
		\end{aligned}
	\end{equation}
\end{theorem}

\begin{proof}
	Insert the finite-volume, ideal projected, and continuous reduced solutions:
	\begin{equation}
		\mathcal{I}_hp(t_n)-Qa^n = \left( \mathcal{I}_hp(t_n)-\mathbf{p}_h(t_n)\right) 
		+ \left( \mathbf{p}_h(t_n)-\mathbf{p}_Q(t_n) \right)
		+ Q\left( a_0(t_n)-\widetilde a(t_n)\right)
		+ Q\left( \widetilde a(t_n)-a^n\right).
	\end{equation}
	The triangle inequality gives
	\begin{multline}
		\left\| \mathcal{I}_h p(t_n) - Qa^n \right\|_h
		\leq \underbrace{\left\| \mathcal{I}_h p(t_n) - \mathbf{p}_h(t_n) \right\|_h}_{\text{Proposition~\ref{prop:FV_error}}}
		+ \underbrace{\left\| \mathbf{p}_h(t_n) - \mathbf{p}_Q(t_n) \right\|_h}_{\text{Theorem~\ref{theo:DDRF_error}}} \\
		+ \underbrace{\left\| Q \left( a_0(t_n)-\widetilde a(t_n)\right)\right\|_h}_{\text{Theorem~\ref{theo:continuous_errors_regularization_shift}}}
		+ \underbrace{\left\| Q \left( \widetilde a(t_n) - a^n \right) \right\|_h}_{\text{Theorem~\ref{theo:second_CN_error}}}.
	\end{multline}
	On the uniform mesh, $\|\mathbf v\|_h=\omega_h^{1/2}\|\mathbf v\|_2$ and $\|Qz\|_2=\|z\|_2$. Apply Proposition~\ref{prop:FV_error}, Theorem~\ref{theo:DDRF_error}, Eq.~\eqref{eq:DDRF_projection_bound}, Theorem~\ref{theo:continuous_errors_regularization_shift}, and Theorem~\ref{theo:second_CN_error} to the four terms.
\end{proof}

\begin{corollary}[Steady-lifting estimate]\label{coro:affine_total_error}
	Suppose the continuous and semi-discrete affine problems admit time-independent steady liftings $\bar p$ and $\bar{\mathbf p}_h$, and let $\bar a$ satisfy
	\begin{equation}
		\widetilde B\bar a+Q^\top\mathbf b_h=0.
	\end{equation}
	Apply Theorem~\ref{theo:total_error} to the homogeneous deviations $u=p-\bar p$, $\mathbf u_h=\mathbf p_h-\bar{\mathbf p}_h$, and $c^n=a^n-\bar a$. If $\mathfrak R_n[\mathbf u_h,c]$ denotes its right-hand side, then
	\begin{equation}
		\|\mathcal I_hp(t_n)-Q(\bar a+c^n)\|_h
		\leq \mathfrak R_n[\mathbf u_h,c]+\varepsilon_{\mathrm{lift},h},
	\end{equation}
	where
	\begin{equation}
		\varepsilon_{\mathrm{lift},h}:=\|\mathcal I_h\bar p-Q\bar a\|_h
		\leq \|\mathcal I_h\bar p-\bar{\mathbf p}_h\|_h
		+\|\bar{\mathbf p}_h-Q\bar a\|_h.
	\end{equation}
\end{corollary}

\begin{remark}[Scope of the estimate]\label{rem:total_error_scope}
Theorem~\ref{theo:total_error} estimates cell averages on the finite-volume grid, with constants that may depend on mesh regularity, coefficient bounds, and reduced-operator norms. The closure quantity $\chi_Q$ is evaluated along the full-order trajectory. For pure Neumann data, the affine corollary uses the compatibility condition and a fixed pressure normalization.
\end{remark}

\section{Numerical Results}\label{sec:results}
The numerical experiments assess pressure accuracy, component-level robustness, performance for high-contrast permeability fields, three-dimensional scalability, and the correspondence between the theoretical error estimates and their realized numerical contributions. Unless otherwise stated, all errors are measured against finite-volume reference solutions on the same computational domain $\Omega$, with harmonic averaging used for the permeability coefficient. The quantitative assessment combines final-time pressure errors and the one-step stability radius with a componentwise error decomposition, a three-grid Richardson estimate of the finite-volume contribution, and a temporal refinement study.

The central low-permeability benchmark supports the baseline and ablation studies, the two high-permeability configurations examine the material-aware random-feature representation under different coefficient structures, and the three-dimensional extension examines the cost of the reduced workflow as the spatial dimension increases. Finally, the error-decomposition experiment evaluates the realized contributions of the principal error components, compares them with the analytical bounds, and examines the observed temporal convergence rate of Crank–Nicolson integration.

\subsection{Benchmark with a central low-permeability inclusion}
\label{sec:central_low_perm}

We first consider a central square inclusion occupying $\Omega_{\mathrm{in}}=\{(x,y):0.4<x<0.6,\ 0.4<y<0.6\}$. The coefficient field, storage coefficient, initial condition, and boundary condition are
\begin{equation}
	A(x,y)=
	\begin{cases}
		0.2, & (x,y)\in\Omega_{\mathrm{in}},\\
		1,   & (x,y)\in\Omega\setminus\Omega_{\mathrm{in}},
	\end{cases}
	\qquad
	S=1,
	\qquad
	p_0(x,y)=\sin(\pi x)\sin(\pi y),
	\qquad
	p|_{\partial\Omega}=0.  \notag
\end{equation}
The central low-permeability experiments use a $65\times65$ grid, $T=0.05$, and $\Delta t=10^{-4}$. Baseline and ablation errors are evaluated against an implicit-Euler finite-volume reference with $\Delta t_{\rm ref}=10^{-4}$; the reduced evolution uses Crank--Nicolson. Appendix~\ref{app:ablation_full} also reports errors against a refined Crank--Nicolson reference.

The transient snapshots in Figure~\ref{fig:sol_central_low_perm_snapshot} show that the predicted pressure follows the reference decay and retains the distortion induced by the inclusion. The largest visible discrepancies remain concentrated within and around the low-permeability region. At the final time, the relative errors are $E_p^{L^2}=9.54\times10^{-4}$ and $E_p^{L^\infty}=1.08\times10^{-3}$.

\begin{figure}[pos=htbp]
	\centering
	\includegraphics[width=0.90\linewidth]{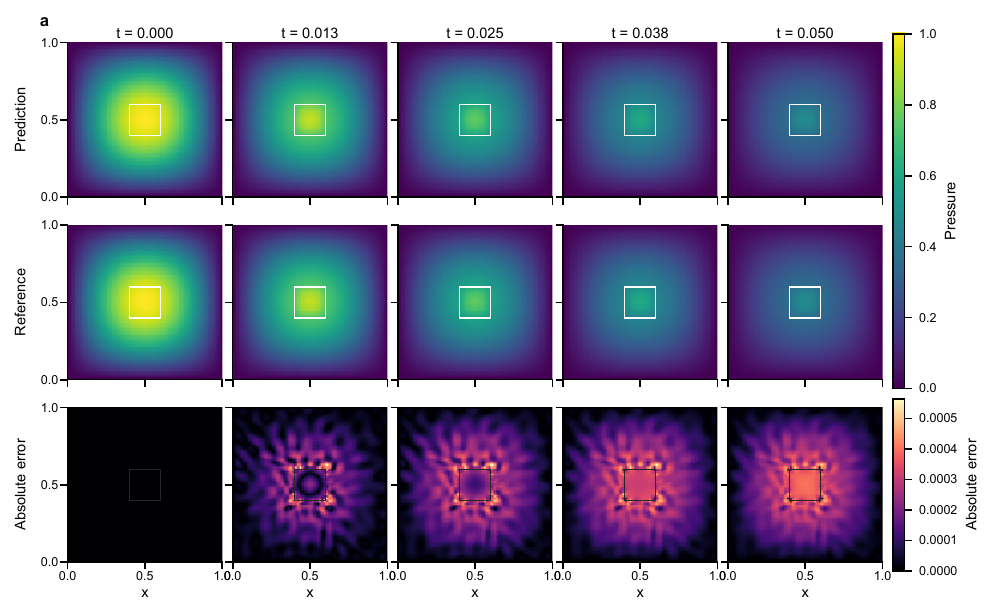}
	\caption{Central low-permeability benchmark: finite-volume reference, DD RF-EDNN pressure, and absolute error at selected times ($65\times65$ grid).}
	\label{fig:sol_central_low_perm_snapshot}
\end{figure}

We next compare the proposed method with a global RF-EDNN counterpart, an RFM solver~\cite{chen2022bridging,rahimi2007random}, PINN~\cite{raissi2019physics}, and XPINN~\cite{jagtap2020extended}. Figure~\ref{fig:baseline_final} presents the final-time spatial comparison; the PINN and XPINN training settings are documented in Appendix~\ref{app:pinn_xpinn_hyperparameters}.

\begin{figure}[pos=htbp]
	\centering
	\includegraphics[width=1.0\linewidth]{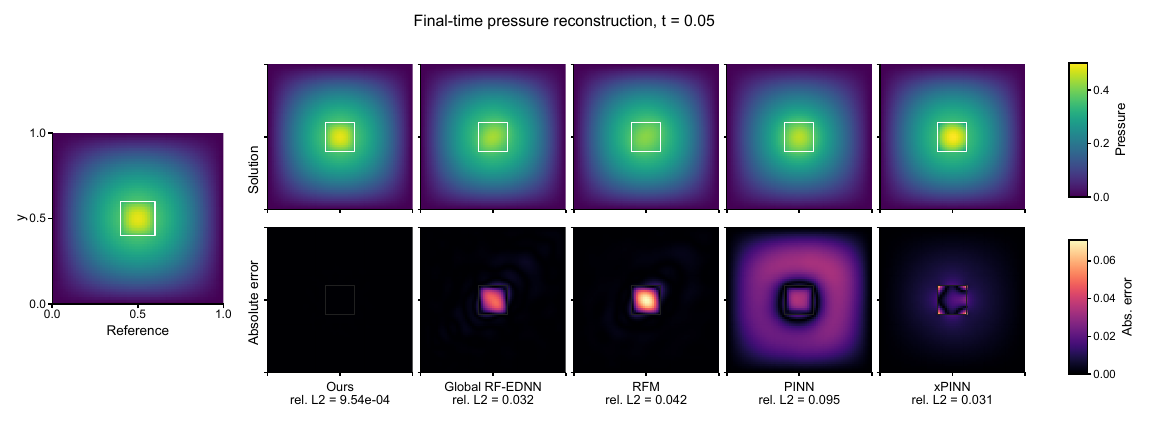}
	\caption{Final-time pressure fields and absolute errors for DD RF-EDNN and four baselines, with the finite-volume reference ($T=0.05$).}
	\label{fig:baseline_final}
\end{figure}

The global RF-EDNN and XPINN recover the overall decay pattern, with their largest discrepancies near the low-permeability region. The RFM and PINN predictions show broader spatial deviations. Table~\ref{tab:central_low_perm_baseline} reports the smallest final-time error for DD RF-EDNN among the tested implementations, more than one order of magnitude below the global RF-EDNN and XPINN errors.

\begin{table}[pos=htbp]
	\centering
	\caption{Final-time baseline comparison for the central low-permeability benchmark ($65\times65$ grid).}
	\label{tab:central_low_perm_baseline}
	\begin{tabular}{lccccc}
		\toprule
		Method & Dimension & $E_p^{L^2}$ & $E_p^{L^\infty}$ & $\rho_{\Delta t}$ & Runtime (s) \\
		\midrule
		Ours & 469 & $9.54\times10^{-4}$ & $1.08\times10^{-3}$ & 0.9981 & 11.11 \\
		Global RF-EDNN & 600 & $3.16\times10^{-2}$ & $1.02\times10^{-1}$ & 1.0000 & 8.76 \\
		RFM & 600 & $4.25\times10^{-2}$ & $1.47\times10^{-1}$ & 0.9980 & 1.57 \\
		PINN & -- & $9.52\times10^{-2}$ & $7.10\times10^{-2}$ & -- & 69.41 \\
		XPINN & -- & $3.11\times10^{-2}$ & $1.30\times10^{-1}$ & -- & 165.52 \\
		\bottomrule
	\end{tabular}
\end{table}

For \texttt{Ours}, Table~\ref{tab:central_low_perm_baseline} reports the median of five measured executions after one warm-up; the observed range was 10.66--12.78~s. The other entries are single-execution timings. Appendix~\ref{app:ablation_full} specifies the measured workflow.

Following the baseline comparison, we examine the contributions of the main algorithmic components through ablation studies on the same benchmark.

\subsection{Ablation and Robustness Analysis}
\label{sec:ablation_robustness}

The ablation study uses the central low-permeability benchmark without changing its physical data. It isolates the effects of pressure-space orthogonalization, recurrent correction, random-feature sampling, feature-space size, coefficient contrast, and the grid/time-step setting.

\begin{figure}[pos=t]
	\centering
	\includegraphics[width=0.90\linewidth]{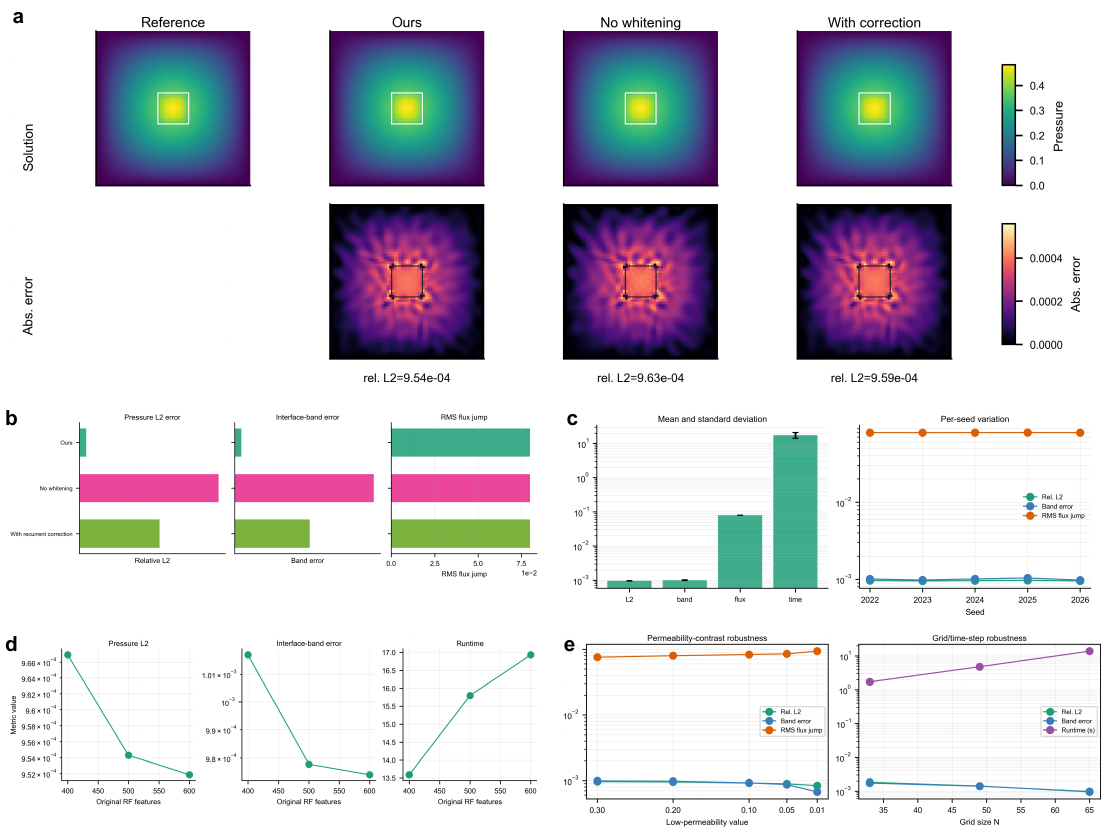}
	\caption{Central-benchmark ablation and robustness: whitening, recurrent correction, random seed, feature-space configuration, permeability contrast, and grid/time-step settings.}
	\label{fig:ablation}
\end{figure}

Figure~\ref{fig:ablation} summarizes the study, Table~\ref{tab:ablation_robustness} gives representative cases, and Appendix~\ref{app:ablation_full} reports the complete results. Orthogonalization changes the coordinate condition number from $\kappa_2(\Phi_r)=6.82\times10^8$ to $\kappa_2(Q)=1$ up to round-off error, with $Q^\top Q=I$. The corresponding final-time relative $L^2$ errors are $9.63\times10^{-4}$ and $9.54\times10^{-4}$.

The final-time relative $L^2$ error is $9.59\times10^{-4}$ with recurrent correction and $9.54\times10^{-4}$ in the default configuration. The default configuration uses fixed-coordinate evolution with recurrent correction disabled.

\begin{table}[pos=htbp]
	\centering
	\caption{
		Representative final-time ablation and robustness results for the central low-permeability benchmark.
	}
	\label{tab:ablation_robustness}
	\begin{tabular}{lcccc}
		\toprule
		Setting & Dimension & $E_p^{L^2}$ & $E_p^{L^\infty}$ & $\rho_{\Delta t}$ \\
		\midrule
		Ours & 469 & $9.54\times10^{-4}$ & $1.08\times10^{-3}$ & 0.9981 \\
		No whitening & 469 & $9.63\times10^{-4}$ & $1.15\times10^{-3}$ & 1.0000 \\
		With recurrent correction & 469 & $9.59\times10^{-4}$ & $1.08\times10^{-3}$ & 0.9981 \\
		Small RF space & 409 & $9.67\times10^{-4}$ & $1.34\times10^{-3}$ & 0.9981 \\
		Rich RF space & 528 & $9.52\times10^{-4}$ & $1.11\times10^{-3}$ & 0.9981 \\
		Strong contrast, $A_{\rm in}=0.01$ & 469 & $8.37\times10^{-4}$ & $5.44\times10^{-4}$ & 0.9995 \\
		Coarse grid, $33\times33$ & 349 & $1.83\times10^{-3}$ & $1.59\times10^{-3}$ & 0.9962 \\
		\bottomrule
	\end{tabular}
\end{table}

Random-feature robustness is assessed over five seeds. The resulting final-time errors satisfy
\begin{equation}
	\operatorname{mean}\!\left(E_p^{L^2}\right)=9.61\times10^{-4},
	\qquad
	\operatorname{std}\!\left(E_p^{L^2}\right)=6.69\times10^{-6}. \notag
\end{equation}
The small, default, and rich configurations contain $M_0=440$, $560$, and $680$ original features and retain $k=409$, $469$, and $528$ pressure coordinates, respectively. Their final-time relative $L^2$ errors range from $9.52\times10^{-4}$ to $9.67\times10^{-4}$. The localization length is 0.055 in the small and default configurations and 0.050 in the rich configuration.

The remaining tests vary the physical contrast and discretization. Decreasing the inclusion permeability from $A=0.3$ to $A=0.01$ keeps the final-time relative $L^2$ error below $10^{-3}$ and gives $\rho_{\Delta t}<1$ throughout the tested range. Refining the grid from $33\times33$ to $65\times65$, together with the reported time-step settings, reduces the pressure error from $1.83\times10^{-3}$ to $9.54\times10^{-4}$.

\subsection{High-contrast permeability configurations}
\label{sec:high_contrast}
We next consider isolated high-conductivity blocks and connected stripe-like pathways, both with $K_f/K_m=1000$. Figure~\ref{fig:high_contrast_permeability} shows the two coefficient fields; in both cases, the material mask is obtained directly from the permeability data.

\begin{figure}[pos=htbp]
	\centering
	\includegraphics[width=0.95\linewidth]{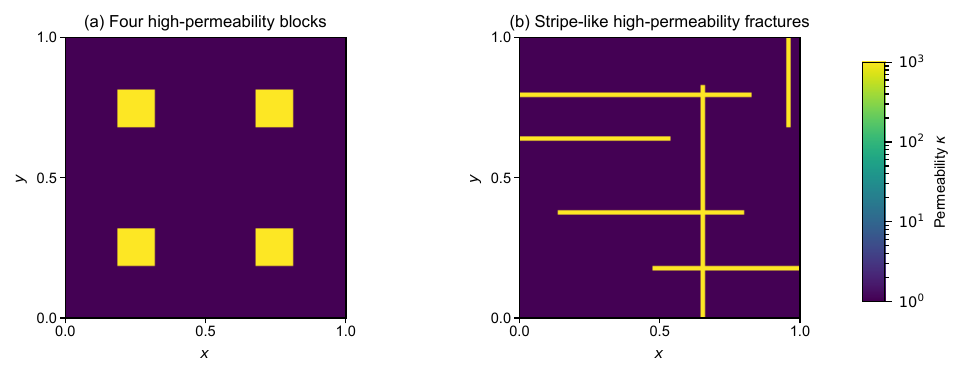}
	\caption{High-contrast permeability fields ($K_m=1$, $K_f=1000$): four square inclusions (left) and six stripe-like channels (right). Color scales are logarithmic.}
	\label{fig:high_contrast_permeability}
\end{figure}

The first configuration contains four high-permeability square inclusions embedded in a matrix with $K_m=1$. Its numerical parameters are grouped as
\begin{equation}
	K_f=1000,
	\qquad S=50,
	\qquad T=0.2,
	\qquad \Delta t=5\times10^{-4},
	\qquad N_x=N_y=65. \notag
\end{equation}
The increased storage coefficient slows the transient response over the longer time interval, and homogeneous Dirichlet data are imposed on $\partial\Omega$.

\begin{figure}[pos=htbp]
	\centering
	\includegraphics[width=1.0\linewidth]{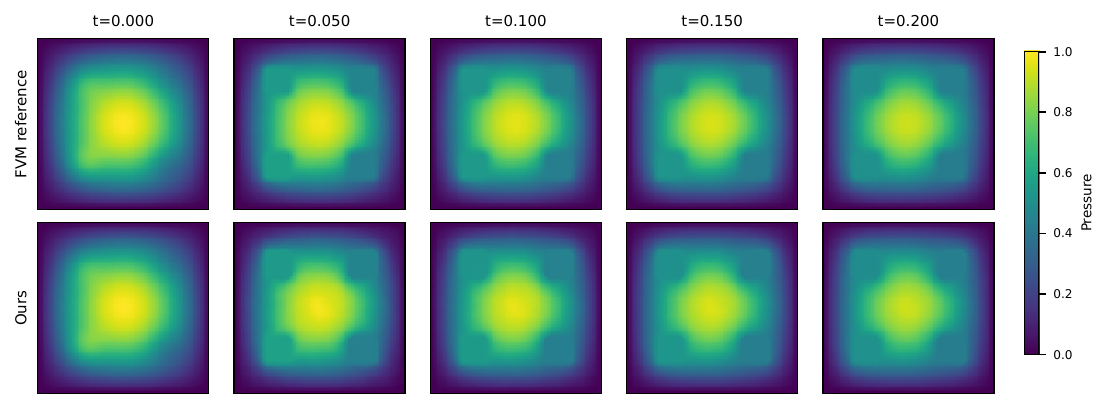}
	\caption{
		Four high-permeability block benchmark ($K_m=1$, $K_f=1000$, and $S=50$). Finite-volume reference (top) and DD RF-EDNN (bottom) prediction at representative times on the $65\times65$ grid.
	}
	\label{fig:four_inclusion_snapshot}
\end{figure}

The second configuration follows the channel-field construction used in KLE-based Darcy-flow datasets~\cite{li2026hybrid}. Six horizontal or vertical stripes are sampled with $K_m=1$ and $K_f=1000$. For cell centers $x_i$, the high-conductivity mask is initialized by
\begin{equation}
	\mathcal{M}_f=\left\{x_i:K(x_i)>K_{\mathrm{th}}\right\},
	\qquad
	K_{\mathrm{th}}=\sqrt{K_mK_f}, \notag
\end{equation}
and is then expanded by one grid layer to include neighboring cells. The same threshold-and-expansion rule is applied to the four-block field.

For the stripe test, $S=1$. To satisfy the nonhomogeneous left--right boundary data at $t=0$ while retaining a nontrivial interior perturbation, the initial pressure used in the computation is
\begin{equation}
	p_0(x,y)=1-x+0.25x(1-x)\sin^2(\pi y)\widehat\psi(x,y),
	\label{eq:stripe_initial_condition}
\end{equation}
where
\begin{equation}
	\widehat\psi(x,y)
	=
	\frac{g_1(x,y)+0.7g_2(x,y)}
	{\displaystyle\max_{(x_i,y_j)\in\Omega_h}
	\left[g_1(x_i,y_j)+0.7g_2(x_i,y_j)\right]},
	\notag
\end{equation}
with
\begin{equation}
	g_1=\exp\!\left[-\frac{(x-0.28)^2+(y-0.32)^2}{0.018}\right],
	\qquad
	g_2=\exp\!\left[-\frac{(x-0.72)^2+(y-0.66)^2}{0.026}\right].
	\notag
\end{equation}
Denoting the left, right, top, and bottom boundaries by $\Gamma_L$, $\Gamma_R$, $\Gamma_T$, and $\Gamma_B$, respectively, the boundary data are
\begin{equation}
	p|_{\Gamma_L}=1,
	\qquad
	p|_{\Gamma_R}=0,
	\qquad
	A\nabla p\cdot n=0
	\quad\text{on }\Gamma_T\cup\Gamma_B. \notag
\end{equation}
The problem is advanced to $T=0.05$ with $\Delta t=10^{-4}$ on a $65\times65$ grid.

\begin{figure}[pos=htbp]
	\centering
	\includegraphics[width=1.0\linewidth]{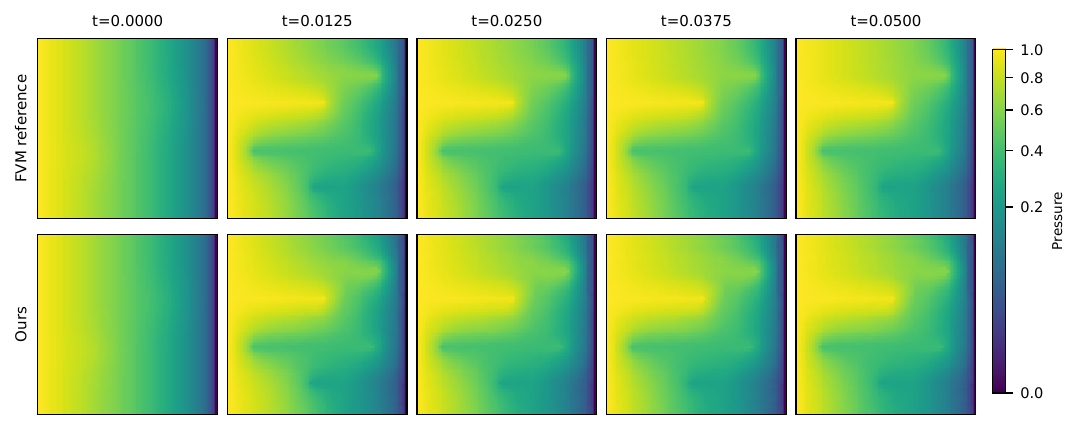}
	\caption{
		Stripe-like high-permeability fracture benchmark ($K_m=1$, $K_f=1000$, $n_f=6$, and $S=1$). Finite-volume reference and DD RF-EDNN prediction at representative times under left--right Dirichlet and top--bottom no-flow conditions.
	}
	\label{fig:stripe_fracture_snapshot}
\end{figure}

Figures~\ref{fig:four_inclusion_snapshot} and~\ref{fig:stripe_fracture_snapshot} show that the reduced solutions retain the principal pressure structures for both high-contrast fields. The final-time errors in Table~\ref{tab:high_contrast_results} are larger than in the central low-permeability benchmark, with the connected stripe field giving the larger $L^2$ and $L^\infty$ errors. Both reduced updates satisfy $\rho_{\Delta t}<1$.

\begin{table}[pos=htbp]
	\centering
	\setlength{\tabcolsep}{4pt}
	\small
	\caption{
		Final-time results for the high-contrast permeability tests.
	}
	\label{tab:high_contrast_results}
	\begin{tabular}{lcccccccc}
		\toprule
		Case & $K_f/K_m$ & $S$ & $T$ & $\Delta t$ & Dim. & $E_p^{L^2}$ & $E_p^{L^\infty}$ & $\rho_{\Delta t}$ \\
		\midrule
		Four high-$K$ blocks & $1000$ & $50$ & $0.2$ & $5.0{\times}10^{-4}$ & $486$ & $1.76{\times}10^{-2}$ & $5.38{\times}10^{-2}$ & 0.9998 \\
		Stripe-like fractures & $1000$ & $1$ & $0.05$ & $1.0{\times}10^{-4}$ & $492$ & $2.33{\times}10^{-2}$ & $8.07{\times}10^{-2}$ & 0.9951 \\
		\bottomrule
	\end{tabular}
\end{table}

\subsection{Three-dimensional scalability}
\label{sec:scalability_3d}

We next examine the three-dimensional applicability and computational cost of DD RF-EDNN on a sequence of increasingly refined Cartesian grids. The computational domain is denoted by $\Omega$, and the permeability coefficient is set to $A=0.2$ inside the central cube and $A=1$ in the surrounding matrix. In the normalized mesh coordinates, the low-permeability cube occupies $0.4<x<0.6$, $0.4<y<0.6$, and $0.4<z<0.6$. The initial pressure is prescribed as

\begin{equation}
	p_0(x,y,z)=\sin(\pi x)\sin(\pi y)\sin(\pi z), \notag
\end{equation}
with homogeneous Dirichlet boundary conditions imposed on $\partial\Omega$. The final time is $T=0.02$, the time step is $\Delta t=5\times10^{-4}$, and the time integration is performed using the Crank--Nicolson scheme.
\begin{figure}[pos=htbp]
	\centering
	\includegraphics[width=1.0\linewidth]{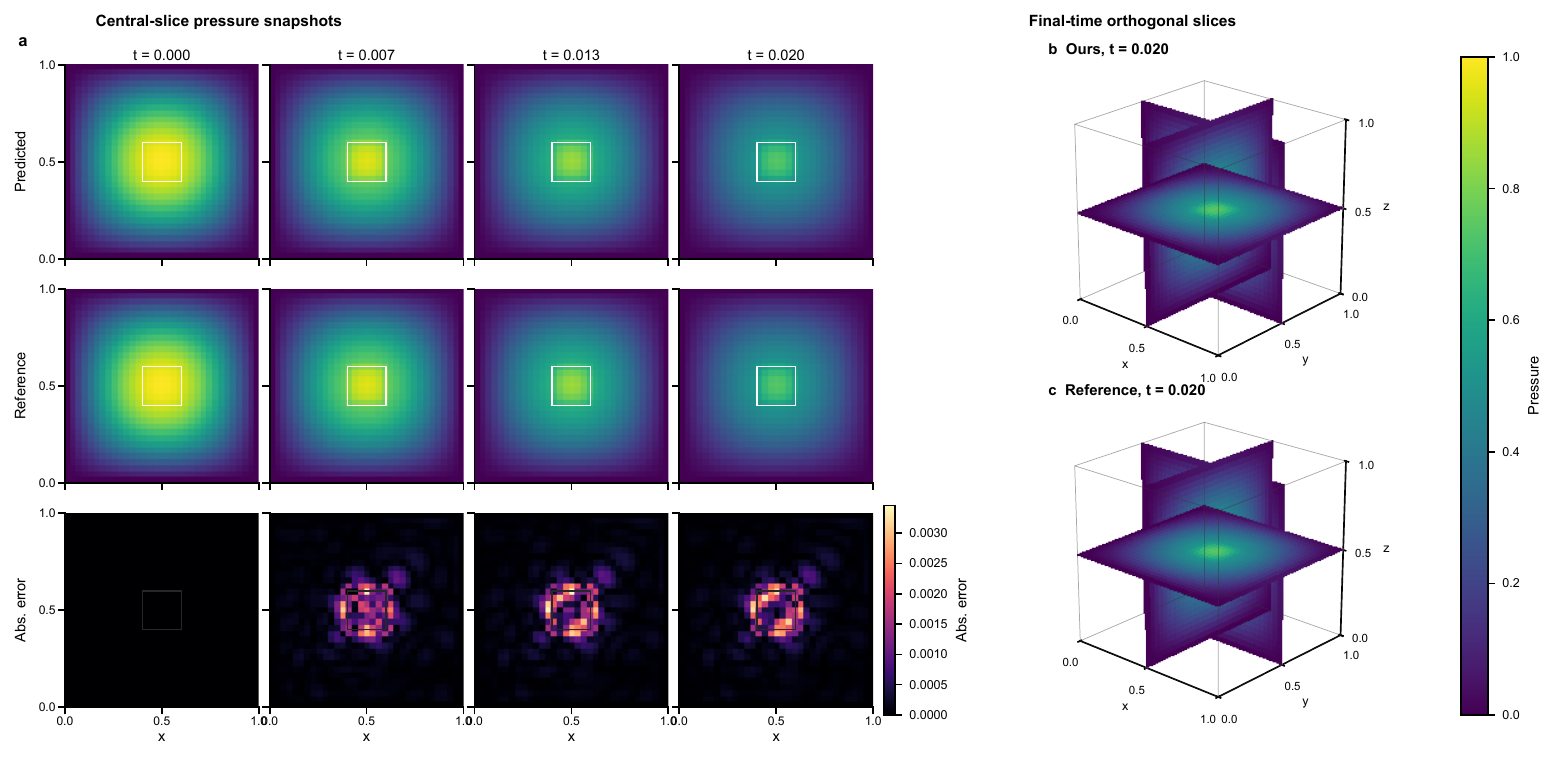}
	\caption{Three-dimensional benchmark on the $33^3$ grid: central-slice reference, prediction, and absolute error over time; final-time orthogonal pressure slices at right.}
	\label{fig:scalability_3d_snapshots}
\end{figure}

On the largest grid, Figure~\ref{fig:scalability_3d_snapshots} shows close agreement between the predicted and reference pressure on the central slice. The visible error remains localized near the central cube, while the final-time orthogonal slices retain the three-dimensional pressure structure.

The scalability study uses grids of $17^3$, $25^3$, and $33^3$, corresponding to 3375, 12167, and 29791 interior degrees of freedom. Table~\ref{tab:scalability_3d} shows that $E_p^{L^2}$ remains on the order of $10^{-3}$ and $\rho_{\Delta t}<1$ at all three resolutions.

Figure~\ref{fig:scalability_3d_composite} separates the accuracy, runtime, and retained-dimension trends. The total runtime increases from 7.70~s to 86.04~s, whereas the reduced time-marching stage remains a small part of the full workflow. Most of the measured cost is associated with feature evaluation, basis construction, and operator projection.

\begin{table}[pos=htbp]
	\centering
	\caption{
		Three-dimensional scalability results at $T=0.02$.
	}
	\label{tab:scalability_3d}
	\begin{tabular}{ccccccc}
		\toprule
		Grid & DoF & Dimension & $E_p^{L^2}$ & $E_p^{L^\infty}$ & $\rho_{\Delta t}$ & Runtime (s) \\
		\midrule
		$17^3$ & 3375  & 500 & $6.57\times10^{-4}$ & $3.10\times10^{-3}$ & 0.9855 & 7.70 \\
		$25^3$ & 12167 & 750 & $1.21\times10^{-3}$ & $6.97\times10^{-3}$ & 0.9855 & 37.48 \\
		$33^3$ & 29791 & 900 & $1.40\times10^{-3}$ & $7.68\times10^{-3}$ & 0.9855 & 86.04 \\
		\bottomrule
	\end{tabular}
\end{table}

\begin{figure}[pos=htbp]
	\centering
	\includegraphics[width=1.0\linewidth]{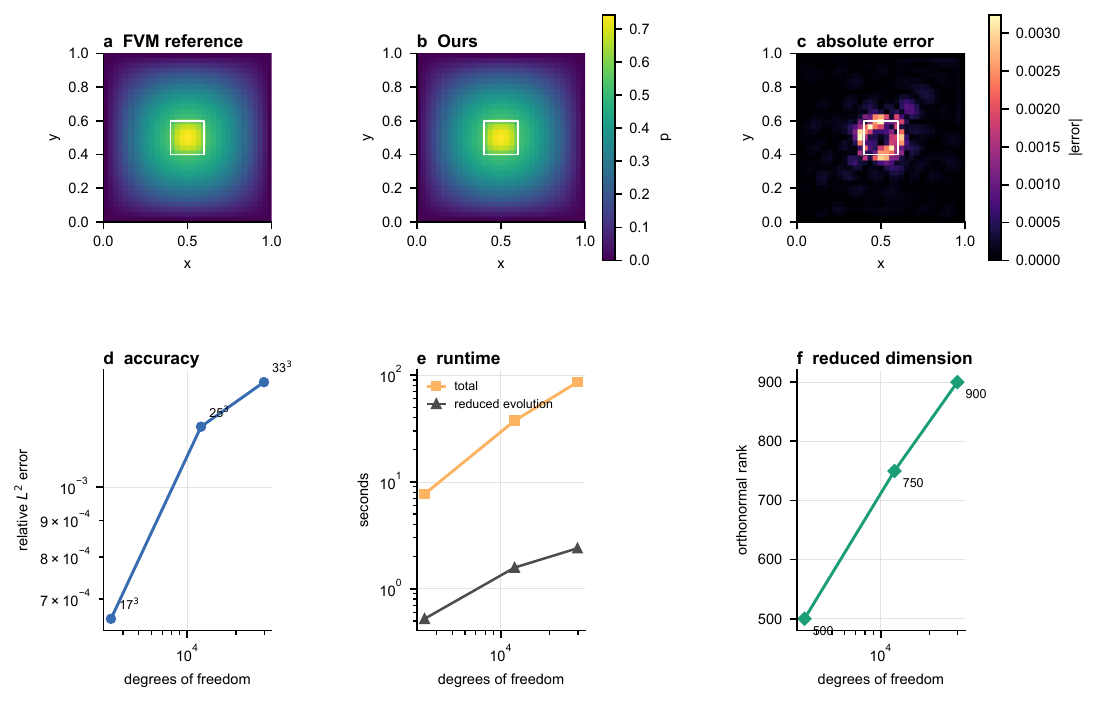}
	\caption{Three-dimensional benchmark: final-time reference, prediction, and absolute error (top); pressure error, runtime, and retained dimension across grid sizes (bottom).}
	\label{fig:scalability_3d_composite}
\end{figure}

\subsection{Error Decomposition and Numerical Verification}
\label{sec:error_decomposition_validation}

\begin{figure}[pos=htbp]
	\centering
	\includegraphics[width=1.0\linewidth]{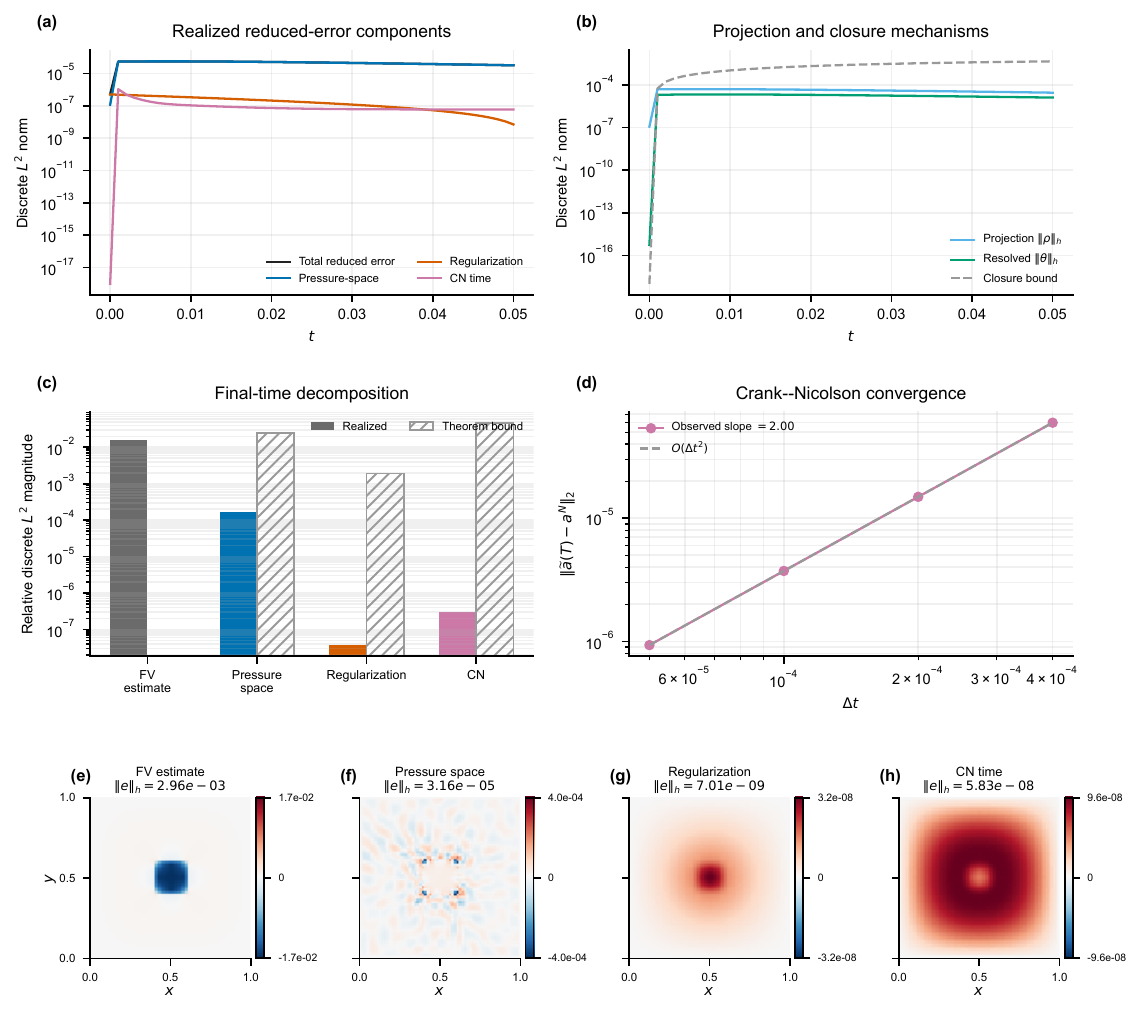}
	\caption{Error decomposition: (a) temporal components; (b) projection and closure; (c) final-time errors and bounds; (d) Crank--Nicolson refinement; (e--h) signed component maps with independent symmetric color scales. The finite-volume term uses Richardson extrapolation.}
	\label{fig:error_decomposition_validation}
\end{figure}

We next quantify the individual contributions in the total error estimate of
Theorem~\ref{theo:total_error} for the central low-permeability benchmark.
The calculation uses the same $65\times65$ grid, final time $T=0.05$, and
time step $\Delta t=10^{-4}$ as the main experiment. The orthogonal pressure
basis contains 469 coordinates, the initial-state and reduced-operator
regularization parameters are $\eta_0=\eta_B=10^{-6}$, and no additional
dissipativity shift is activated, i.e., $\alpha_{\mathrm{diss}}=0$. To isolate the errors introduced
at different stages, the finite-volume semi-discrete trajectory and the two
continuous reduced trajectories are evaluated by matrix exponentials, whereas
the reported fully discrete reduced trajectory is advanced by
Crank--Nicolson.

Effectivity is the analytical bound divided by the realized component magnitude. The estimated total includes the empirical finite-volume contribution.

\begin{table}[pos=ht]
	\centering
	\caption{Final-time normalized error components, analytical bounds, and effectivity indices for the central benchmark.}
	\label{tab:error_decomposition_validation}
	\small
	\begin{tabular}{lccc}
		\toprule
		Component
		& Realized relative magnitude
		& Upper estimate
		& Effectivity \\
		\midrule
		Finite-volume estimate
		& $1.568\times10^{-2}$ & -- & -- \\
		Pressure space and closure
		& $1.670\times10^{-4}$ & $2.461\times10^{-2}$ & $1.47\times10^{2}$ \\
		Regularization and dissipativity shift
		& $3.711\times10^{-8}$ & $1.928\times10^{-3}$ & $5.19\times10^{4}$ \\
		Crank--Nicolson time integration
		& $3.085\times10^{-7}$ & $4.743\times10^{-2}$ & $1.54\times10^{5}$ \\
		Estimated total
		& $1.565\times10^{-2}$ & $8.965\times10^{-2}$ & -- \\
		\bottomrule
	\end{tabular}
\end{table}

We estimate the finite-volume contribution using a continuum surrogate
$\mathbf p_R(T)$ constructed from the nested grids $33\times33$, $65\times65$,
and $129\times129$ by full weighting and three-grid Richardson extrapolation.
At the final time, the resulting component identity is
\begin{equation}
	\widehat{\mathbf e}_{\rm tot}
	=
	\underbrace{\left(\mathbf p_R-\mathbf p_h\right)}_{\widehat{\mathbf e}_{\rm FV}}
	+
	\underbrace{\left(\mathbf p_h-\mathbf p_Q\right)}_{\mathbf e_Q}
	+
	\underbrace{Q\left(a_0-\widetilde a\right)}_{\mathbf e_B}
	+
	\underbrace{Q\left(\widetilde a-a^N\right)}_{\mathbf e_{\Delta t}},
	\label{eq:numerical_error_decomposition}
\end{equation}
where all quantities in Eq.~\eqref{eq:numerical_error_decomposition} are
evaluated at $T$. We report
$\widehat{\mathcal E}_i=\|\mathbf e_i\|_h/\|\mathbf p_R\|_h$ for each
component. The two successive grid discrepancies are
$4.16\times10^{-3}$ and $1.73\times10^{-3}$, giving an observed spatial
rate of $1.27$. The Richardson result provides an empirical estimate of
the finite-volume discretization contribution.

Figure~\ref{fig:error_decomposition_validation}(a) shows that the
pressure-space term is the dominant error introduced by the reduced model
throughout most of the transient interval. At $T=0.05$, its relative
magnitude is $1.67\times10^{-4}$, whereas the realized regularization and
Crank--Nicolson contributions are $3.711\times10^{-8}$ and
$3.085\times10^{-7}$, respectively. Once the estimated finite-volume term is
included, the estimated total relative error becomes
$1.565\times10^{-2}$ and is dominated by the spatial discretization estimate.
Cancellation among the signed vectors in
Eq.~\eqref{eq:numerical_error_decomposition} makes the total norm slightly
smaller than the finite-volume component norm.

All three analytical estimates exceed their corresponding realized
components, in agreement with Theorems~\ref{theo:DDRF_error},%
~\ref{theo:continuous_errors_regularization_shift}, and~\ref{theo:second_CN_error}.
The effectivity indices in Table~\ref{tab:error_decomposition_validation}
quantify the conservativeness of these bounds, whose derivations use triangle
inequalities, contractivity estimates, and time-uniform derivative bounds.
The spatial maps in
Figure~\ref{fig:error_decomposition_validation}(e--h) further distinguish
the mechanisms: the estimated finite-volume error is concentrated in the
low-permeability inclusion, the pressure-space term is localized near the
coefficient transition, and the two remaining components have substantially
smaller absolute magnitudes.

The algebraic diagnostics agree with the formal decomposition. In particular,
$\|Q^\top Q-I\|_2=5.80\times10^{-15}$, while the operator identity, initial
identity, and final vector-decomposition residuals are zero to the reported
precision. Finally, the temporal refinement study with
$\Delta t\in\{4,2,1,0.5\}\times10^{-4}$ gives an observed order of $2.00$,
as shown in Figure~\ref{fig:error_decomposition_validation}(d), providing a
direct numerical verification of the second-order Crank--Nicolson estimate.

\section{Discussion}\label{sec:discussion}

The material-adapted random-feature construction provides a trial space for reduced transient Darcy dynamics. The spatial basis is generated from the permeability field and randomized features before a transient reference trajectory is computed. Projection of the conservative finite-volume operator then yields an initial-value problem for orthogonal pressure coordinates. Across the central low-permeability benchmark, two high-conductivity configurations, and the three-dimensional test, this construction retained the principal pressure structures and produced spectral radii below one. Material-aware feature allocation connects the spatial representation to the coefficient field.

The central-benchmark comparison links the observed accuracy to the material-adapted pressure space. Its final-time error was more than one order of magnitude below those of the global RF-EDNN and RFM configurations in Table~\ref{tab:central_low_perm_baseline}. Orthogonalization reduced the coordinate condition number from $6.82\times10^8$ to approximately one, while the relative $L^2$ error changed from $9.63\times10^{-4}$ to $9.54\times10^{-4}$. The orthonormal basis also supports reduced-operator construction and the stability analysis. Recurrent correction gave $9.59\times10^{-4}$ and was disabled in the default configuration. The seed and feature-space studies produced closely grouped errors over the tested settings.

The stability analysis identifies the mechanism behind the contractive update. Orthogonal projection preserves the nonpositive symmetric part of a dissipative finite-volume generator, as stated in Theorem~\ref{theo:reduced_stability}. The logarithmic-norm correction enforces the prescribed dissipativity margin when needed. The operators in the error-decomposition and high-contrast experiments already satisfied this condition, so their shifts were zero. Under the resulting dissipativity condition, Crank--Nicolson is contractive by Theorem~\ref{theo_well_posedness}.

The componentwise error experiment separates two priorities for improving accuracy. The pressure-space and closure contribution dominates the error introduced by reduction, while the Richardson finite-volume estimate dominates the estimated total error. Regularization, dissipativity correction, and Crank--Nicolson integration contribute substantially less in the tested configuration. These results identify grid refinement, feature enrichment near large projection or closure residuals, and retained-rank adaptation as targets for further development. The observed second-order temporal convergence supports this emphasis on spatial approximation.

The four-block and stripe-like benchmarks extend the material-adapted construction to permeability fields with $K_f/K_m=1000$. Their masks are obtained directly from the permeability data, and the reduced solutions retain the dominant pressure patterns of isolated blocks and connected pathways. The respective relative $L^2$ errors were $1.76\times10^{-2}$ and $2.33\times10^{-2}$. The two benchmarks assess the method under different storage coefficients, final times, initial states, and boundary conditions.

The three-dimensional experiment extends the construction to $29{,}791$ interior degrees of freedom, with relative $L^2$ errors on the order of $10^{-3}$ and spectral radii below one at all three resolutions. Feature evaluation, basis construction, and operator projection dominate the measured cost; reduced time marching accounts for a small fraction. For repeated queries with the same permeability field, finite-volume operator, and time step, the basis, reduced operator, and Crank--Nicolson factorization can be reused for compatible initial states or boundary histories. Across the tested grids, the retained dimension increased from 500 to 900 and the total runtime from 7.70~s to 86.04~s.

The analysis covers uniform Cartesian finite-volume discretizations with constant storage and piecewise-constant scalar permeability. The numerical tests use prescribed synthetic coefficient fields. Future work will investigate more efficient basis construction, adaptive feature allocation, irregular grids, time-dependent forcing, and larger three-dimensional geological configurations.

\section{Conclusions}\label{sec:conclusions}

This work developed a material-adapted random-feature reduced framework for transient Darcy problems with discontinuous and high-contrast permeability fields. The pressure space is constructed directly from the permeability distribution by assigning randomized spatial features to distinct material regions. A two-stage singular-value procedure then removes redundant directions and produces an orthogonal pressure basis. Projection of the conservative finite-volume generator onto this basis yields a compact initial-value problem for the pressure coordinates, while the spatial random features remain fixed. The resulting system can be advanced by standard temporal schemes; Crank--Nicolson gives a second-order contractive update for a dissipative reduced operator. This construction avoids a database of full-order transient snapshots during basis generation and retains a direct algebraic connection to the underlying finite-volume dynamics.

The theoretical and numerical results provide complementary evidence for the proposed construction. The analysis establishes inherited dissipativity of the projected operator, a deterministic DD-RF compression bound, unconditional contractivity of the Crank--Nicolson update, and a conditional grid-level error decomposition. The latter separates finite-volume consistency, dictionary and SVD truncation, dynamical closure, operator perturbation, and temporal integration; a steady-lifting corollary accounts for time-independent affine data. On the central low-permeability benchmark, the method attained a final-time relative $L^2$ error of $9.54\times10^{-4}$ and gave the lowest error among the tested implementations. The ablation study showed that orthogonalization primarily improves coordinate conditioning, while the pressure accuracy remained robust over the tested random seeds and feature counts. The two configurations with $K_f/K_m=1000$ retained the principal pressure structures for isolated blocks and connected conductive pathways, with relative errors on the order of $10^{-2}$. The three-dimensional study maintained errors on the order of $10^{-3}$ up to 29,791 interior degrees of freedom. Finally, the numerical error decomposition identified the pressure space as the dominant source of reduced-model error and the empirical finite-volume contribution as the dominant term in the complete error budget. The observed temporal order of $2.00$ was consistent with the Crank--Nicolson estimate.

DD RF-EDNN connects material-adapted random-feature approximation with conservative finite-volume dynamics in a reduced evolution formulation. Future developments will focus on reducing offline basis-construction costs, improving adaptive feature allocation and rank selection, and extending the framework to more general discretizations and three-dimensional subsurface flow problems.

\appendix
\section{Auxiliary Proofs}
\label{app:auxiliary_proofs}

This appendix collects routine algebraic arguments used in Section~\ref{sec:theory}. Their statements remain in the main text so that the analytical dependence is visible without interrupting the principal stability and error estimates.

\subsection{Orthogonal projection}

\begin{proof}[Proof of Lemma~\ref{lemma:orthogonal_projection}]
Using $Q^\top Q=I_k$ gives $P_Q^2=QQ^\top QQ^\top=P_Q$ and $P_Q^\top=P_Q$. Hence $P_Q$ is the Euclidean orthogonal projector onto $\mathcal V_Q$. For any $\mathbf w\in\mathcal V_Q$,
\begin{equation}
	\|\mathbf v-\mathbf w\|_2^2
	=
	\|\mathbf v-P_Q\mathbf v\|_2^2
	\|P_Q\mathbf v-\mathbf w\|_2^2.
\end{equation}
The second term is nonnegative and vanishes for $\mathbf w=P_Q\mathbf v$, proving the best-approximation identity.
\end{proof}

\subsection{Ridge formulas}

\begin{proof}[Proof of Lemma~\ref{lemma:regularized_initial_projection}]
The gradient of
\begin{equation}
	\mathcal J_0(a)=\|Qa-\mathbf p_h^0\|_2^2+\eta_0\|a\|_2^2
\end{equation}
is $2(1+\eta_0)a-2Q^\top\mathbf p_h^0$. Its positive-definite Hessian gives the unique minimizer
\begin{equation}
	a_{\eta_0}^0=(1+\eta_0)^{-1}Q^\top\mathbf p_h^0.
\end{equation}
Therefore $Qa_{\eta_0}^0=(1+\eta_0)^{-1}P_Q\mathbf p_h^0$, which yields the stated regularization error.
\end{proof}

\begin{proof}[Proof of Lemma~\ref{lemma:operator_regression}]
With $Q^\top Q=I_k$, differentiation of the Frobenius-norm objective gives
\begin{equation}
	\nabla_B\mathcal J_B(B)=2(1+\eta_B)B-2Q^\top L_hQ.
\end{equation}
The Hessian is positive definite in vectorized coordinates. Thus the unique stationary point is
\begin{equation}
	B_{\eta_B}=(1+\eta_B)^{-1}Q^\top L_hQ=(1+\eta_B)^{-1}B_0.
\end{equation}
\end{proof}

\subsection{Auxiliary stability bounds}

\begin{proof}[Proof of Corollary~\ref{coro:semigroup_contractivity}]
Apply the energy argument in Theorem~\ref{theo:reduced_stability} to $B_0$, $B_{\eta_B}$, and $\widetilde B$. Taking the supremum over nonzero initial data gives the three induced-norm bounds. If $\widetilde Bv=\lambda v$, then
\begin{equation}
	2\operatorname{Re}(\lambda)\|v\|_2^2
	=v^*(\widetilde B+\widetilde B^*)v
	\leq-2\delta\|v\|_2^2,
\end{equation}
and hence $\operatorname{Re}(\lambda)\leq-\delta$.
\end{proof}

\begin{proof}[Proof of Corollary~\ref{coro:closure_coupling}]
The induced operator norm gives
\begin{equation}
	\chi_Q(t)
	\leq
	\|P_QL_h(I-P_Q)\|_{2\to2}
	\|(I-P_Q)\mathbf p_h(t)\|_2
	=\gamma_Q\varepsilon_Q(t).
\end{equation}
Substitution into Theorem~\ref{theo:DDRF_error} yields the second estimate.
\end{proof}

\section{Random-Feature Reproducibility}
\label{app:rf_reproducibility}

Tables~\ref{tab:rf_common_rules}--\ref{tab:rf_high_contrast_parameters} specify the random-feature construction for the three two-dimensional benchmarks. Features are sampled once with NumPy's \texttt{default\_rng} and remain fixed during evolution. All linear algebra uses double precision. The central baseline and ablation experiments use an implicit-Euler finite-volume reference at $\Delta t_{\rm ref}=\Delta t=10^{-4}$ and Crank--Nicolson reduced evolution. The refined-reference comparison is reported in Appendix~\ref{app:ablation_full}. In Table~\ref{tab:rf_common_rules}, $m$ is the expanded material mask and $\Phi_s$ contains the first $M_s$ base-feature columns.

\begin{table}[pos=!htbp,abovecap=2pt]
	\centering
	\caption{Random-feature construction and truncation rules for the two-dimensional benchmarks.}
	\label{tab:rf_common_rules}
	\small
	\setlength{\tabcolsep}{4pt}
	\begin{tabular}{p{0.20\textwidth}p{0.70\textwidth}}
		\toprule
		Component & Reproducible implementation rule \\
		\midrule
		Activation and boundary factor &
		$\sigma(s)=\tanh(s)$ and $b_\Omega(x,y)=x(1-x)y(1-y)$. \\
		High-contrast base features &
		$\phi_j(x)=b_\Omega(x)\tanh(w_j^\top x+\beta_j)$, with independent $w_j\sim\mathcal N(0,3^2I_2)$ and $\beta_j\sim\mathcal U[-2,2]$. \\
		Mask enrichment &
		$\Phi_0=[\Phi,\,m\Phi_s,\,(1-m)\Phi_s]$. The first $M_s=\operatorname{round}(\vartheta M)$ base features are duplicated, where $M$ is the base count and $\vartheta$ is the split fraction. \\
		High-contrast mask &
		Cells satisfying $K\geq\sqrt{K_mK_f}$ are selected and the resulting binary mask is expanded by one Manhattan grid layer. \\
		SVD truncation &
		A singular vector is retained when $\sigma_j>\tau_{\rm svd}\sigma_1$; the retained dimension is additionally limited by the case-specific rank cap. \\
		Dissipativity correction &
		$\mu_2(B)=\lambda_{\max}((B+B^\top)/2)$ and $\alpha_{\mathrm{diss}}=\max\{0,\mu_2(B)+\delta\}$ with $\delta=10^{-4}$, followed by $\widetilde B=B-\alpha_{\mathrm{diss}}I$. \\
		\bottomrule
	\end{tabular}
\end{table}

The central low-permeability benchmark assigns feature groups using the square geometry. Inner coordinates are mapped to $((x-0.5)/0.1,(y-0.5)/0.1)$ and outer coordinates to $(2x-1,2y-1)$. Side-localized features form four groups in each material region, giving eight groups in total; their counts in Table~\ref{tab:rf_central_sampling} are per group. Recurrent correction is disabled.

\begin{table}[pos=!htbp,abovecap=2pt]
	\centering
	\caption{Random-feature sampling and compression for the central low-permeability benchmark.}
	\label{tab:rf_central_sampling}
	\small
	\setlength{\tabcolsep}{4pt}
	\begin{tabular}{p{0.28\textwidth}p{0.18\textwidth}p{0.44\textwidth}}
		\toprule
		Quantity & Value & Sampling or implementation detail \\
		\midrule
		Random seed & 2026 & One generator is advanced in the row order shown below. \\
		Inner-region features & 200 & $w\sim\mathcal N(0,1.8^2I_2)$, $\beta\sim\mathcal U[-2,2]$. \\
		Outer-region features & 240 & $w\sim\mathcal N(0,3^2I_2)$, $\beta\sim\mathcal U[-2,2]$. \\
		Side-localized features & $15$ per group & $w_\tau\sim\mathcal N(0,2.5^2)$, $w_n\sim\mathcal N(0,1.5^2)$, $\beta\sim\mathcal U[-2,2]$. \\
		Localization length & $\ell_\Gamma=0.055$ & Gaussian normal window $\exp[-(\rho_\Gamma/\ell_\Gamma)^2]$. \\
		Original feature count & $M_0=560$ & $200+240+8\times15$. \\
		First SVD & $\tau_{\rm svd}=10^{-10}$ & Rank cap $r\leq500$. \\
		Pressure-space SVD & $\tau_Q=10^{-12}$ & Retained pressure dimension $k=469$. \\
		Regularization & $\eta_0=\eta_B=10^{-6}$ & Applied to the initial projection and reduced operator. \\
		Time evolution & Crank--Nicolson & $\Delta t=10^{-4}$, $T=0.05$; recurrent correction disabled. \\
		\bottomrule
	\end{tabular}
\end{table}

The two high-contrast tests use the common masked construction in Table~\ref{tab:rf_common_rules}. Table~\ref{tab:rf_high_contrast_parameters} reports both the prescribed inputs and the resulting retained dimensions. For the stripe field, the same seed is used for permeability generation and feature sampling; the generator for each operation is initialized independently.

\begin{table}[pos=!htbp,abovecap=2pt]
	\centering
	\caption{Case-specific random-feature and evolution parameters for the two high-contrast benchmarks. A dash denotes a deterministic permeability geometry.}
	\label{tab:rf_high_contrast_parameters}
	\footnotesize
	\setlength{\tabcolsep}{4.5pt}
	\begin{tabular}{lcc}
		\toprule
		Parameter & Four high-$K$ blocks & Stripe-like fractures \\
		\midrule
		Grid & $65\times65$ & $65\times65$ \\
		Permeability data & $K_m=1$, $K_f=1000$ & $K_m=1$, $K_f=1000$, $n_f=6$ \\
		Geometry & Centers $\{1/4,3/4\}^2$, half-width $0.075$ & Axis-aligned random channels \\
		Storage coefficient & $S=50$ & $S=1$ \\
		$T$; $\Delta t$ & $0.2$; $5\times10^{-4}$ & $0.05$; $10^{-4}$ \\
		Permeability seed & -- & 2026 \\
		Feature seed & 2041 & 2026 \\
		Base feature count $M$ & 1000 & 900 \\
		Split fraction $\vartheta$ & 0.45 & 0.45 \\
		Duplicated count $M_s$ & 450 & 405 \\
		Total feature count $M_0$ & 1900 & 1710 \\
		Mask expansion & One grid layer & One grid layer \\
		SVD tolerance $\tau_{\rm svd}$ & $10^{-10}$ & $10^{-10}$ \\
		Rank cap & 800 & 700 \\
		Retained dimension $k$ & 486 & 492 \\
		Initial/operator ridge & 0 & 0 \\
		Stability margin $\delta$ & $10^{-4}$ & $10^{-4}$ \\
		Realized dissipativity shift $\alpha_{\mathrm{diss}}$ & 0 & 0 \\
		Time integrator & Crank--Nicolson & Crank--Nicolson \\
		\bottomrule
	\end{tabular}
\end{table}

\FloatBarrier
\section{Complete Ablation and Robustness Results}\label{app:ablation_full}

Tables~\ref{tab:app_ablation_structure}--\ref{tab:app_ablation_grid_time} give the complete ablation and robustness results for Section~\ref{sec:ablation_robustness}. Condition numbers refer to the coordinates used in each variant. In the seed and contrast studies, $E_{\rm band}$ is the relative $L^2$ pressure error in the material-transition neighborhood.

Tables~\ref{tab:app_ablation_structure}--\ref{tab:app_ablation_grid_time} report single-execution wall-clock times. Each timed workflow includes feature construction, compression and whitening where applicable, operator assembly, evolution with pressure reconstruction, and trajectory and post-integration diagnostics. Reference generation, plotting, and file export are excluded. The default configuration (seed 2026, $M_0=560$, $k=469$, $A_{\rm in}=0.2$, $N=65$) is executed separately in each study. Its numerical diagnostics agree across runs, with measured times of 22.87, 15.23, 15.80, 14.01, and 13.84~s in study order.

The runtime for Ours in Table~\ref{tab:central_low_perm_baseline} is based on one warm-up followed by five measured executions on an Intel Core i7-9750H CPU with Windows 10, Python 3.11.14, NumPy 1.26.4, and SciPy 1.17.1. BLAS/OpenMP libraries used one thread. Each execution rebuilt the feature space and reduced operator and performed the workflow defined above. Timing with \texttt{time.perf\_counter} excluded case setup, reference generation, and exports. The median was 11.11~s and the range 10.66--12.78~s. All six executions retained $k=469$ and produced bitwise-identical final pressure arrays.

For the central benchmark, the implicit-Euler reference at $\Delta t_{\rm ref}=10^{-4}$ gives $E_p^{L^2}=9.54\times10^{-4}$ and $E_p^{L^\infty}=1.08\times10^{-3}$, as reported in the baseline and ablation tables. The same reduced pressure evaluated against a Crank--Nicolson reference at $\Delta t_{\rm ref}=2.5\times10^{-5}$ gives $E_p^{L^2}=1.66\times10^{-4}$ and $E_p^{L^\infty}=8.19\times10^{-4}$.

\begin{table}[pos=!htbp,abovecap=2pt]
	\centering
	\caption{Structural ablations on the central low-permeability benchmark.}
	\label{tab:app_ablation_structure}
	\footnotesize
	\setlength{\tabcolsep}{3.2pt}
	\begin{tabular}{lcccccc}
		\toprule
		Setting & Dimension & Cond. & $E_p^{L^2}$ & $E_p^{L^\infty}$ & $\rho_{\Delta t}$ & Runtime (s) \\
		\midrule
		Ours & 469 & $1.00\times10^{0}$ & $9.54\times10^{-4}$ & $1.08\times10^{-3}$ & 0.9981 & 22.87 \\
		No whitening & 469 & $6.82\times10^{8}$ & $9.63\times10^{-4}$ & $1.15\times10^{-3}$ & 1.0000 & 22.99 \\
		With recurrent correction & 469 & $1.00\times10^{0}$ & $9.59\times10^{-4}$ & $1.08\times10^{-3}$ & 0.9981 & 23.24 \\
		\bottomrule
	\end{tabular}
\end{table}

\begin{table}[pos=!htbp,abovecap=2pt]
	\centering
	\caption{Random-seed robustness results for the proposed method on the central low-permeability benchmark.}
	\label{tab:app_ablation_seed}
	\footnotesize
	\setlength{\tabcolsep}{3.2pt}
	\begin{tabular}{ccccccc}
		\toprule
		Seed & Dimension & $E_p^{L^2}$ & $E_p^{L^\infty}$ & $E_{\rm band}$ & $\rho_{\Delta t}$ & Runtime (s) \\
		\midrule
		2022 & 469 & $9.64\times10^{-4}$ & $1.35\times10^{-3}$ & $1.01\times10^{-3}$ & 0.9981 & 22.04 \\
		2023 & 469 & $9.53\times10^{-4}$ & $1.10\times10^{-3}$ & $9.80\times10^{-4}$ & 0.9981 & 20.10 \\
		2024 & 469 & $9.62\times10^{-4}$ & $1.24\times10^{-3}$ & $1.01\times10^{-3}$ & 0.9981 & 14.70 \\
		2025 & 469 & $9.69\times10^{-4}$ & $1.47\times10^{-3}$ & $1.04\times10^{-3}$ & 0.9981 & 15.53 \\
		2026 & 469 & $9.54\times10^{-4}$ & $1.08\times10^{-3}$ & $9.78\times10^{-4}$ & 0.9981 & 15.23 \\
		\bottomrule
	\end{tabular}
\end{table}

Table~\ref{tab:app_ablation_feature} uses $M_0=n_{\rm in}+n_{\rm out}+8n_{\rm side}$ for the original dictionary size and $k$ for the dimension after compression and whitening. The small, default, and rich configurations use first-SVD rank caps of 420, 500, and 600; their localization lengths are 0.055, 0.055, and 0.050, respectively.

\begin{table}[pos=!htbp,abovecap=2pt]
	\centering
	\caption{Feature-space configurations for the central low-permeability benchmark.}
	\label{tab:app_ablation_feature}
	\footnotesize
	\setlength{\tabcolsep}{3.2pt}
	\begin{tabular}{lcccccc}
		\toprule
		Setting & $M_0$ & $k$ & $E_p^{L^2}$ & $E_p^{L^\infty}$ & $\rho_{\Delta t}$ & Runtime (s) \\
		\midrule
		Small RF space & 440 & 409 & $9.67\times10^{-4}$ & $1.34\times10^{-3}$ & 0.9981 & 13.59 \\
		Default RF space & 560 & 469 & $9.54\times10^{-4}$ & $1.08\times10^{-3}$ & 0.9981 & 15.80 \\
		Rich RF space & 680 & 528 & $9.52\times10^{-4}$ & $1.11\times10^{-3}$ & 0.9981 & 16.93 \\
		\bottomrule
	\end{tabular}
\end{table}

\begin{table}[pos=!htbp,abovecap=2pt]
	\centering
	\caption{Permeability-contrast robustness results obtained by varying the inclusion permeability $A_{\rm in}$.}
	\label{tab:app_ablation_contrast}
	\footnotesize
	\setlength{\tabcolsep}{3.2pt}
	\begin{tabular}{ccccccc}
		\toprule
		$A_{\rm in}$ & Dimension & $E_p^{L^2}$ & $E_p^{L^\infty}$ & $E_{\rm band}$ & $\rho_{\Delta t}$ & Runtime (s) \\
		\midrule
		0.30 & 469 & $9.65\times10^{-4}$ & $1.13\times10^{-3}$ & $9.90\times10^{-4}$ & 0.9981 & 14.64 \\
		0.20 & 469 & $9.54\times10^{-4}$ & $1.08\times10^{-3}$ & $9.78\times10^{-4}$ & 0.9981 & 14.01 \\
		0.10 & 469 & $9.20\times10^{-4}$ & $8.91\times10^{-4}$ & $9.22\times10^{-4}$ & 0.9981 & 13.81 \\
		0.05 & 469 & $8.93\times10^{-4}$ & $7.10\times10^{-4}$ & $8.70\times10^{-4}$ & 0.9984 & 15.74 \\
		0.01 & 469 & $8.37\times10^{-4}$ & $5.44\times10^{-4}$ & $6.73\times10^{-4}$ & 0.9995 & 14.10 \\
		\bottomrule
	\end{tabular}
\end{table}

\begin{table}[pos=!htbp,abovecap=2pt]
	\centering
	\caption{Grid and time-step robustness results on the central low-permeability benchmark.}
	\label{tab:app_ablation_grid_time}
	\footnotesize
	\setlength{\tabcolsep}{3.2pt}
	\begin{tabular}{ccccccc}
		\toprule
		$N$ & $\Delta t$ & Dimension & $E_p^{L^2}$ & $E_p^{L^\infty}$ & $\rho_{\Delta t}$ & Runtime (s) \\
		\midrule
		33 & $2.0\times10^{-4}$ & 349 & $1.83\times10^{-3}$ & $1.59\times10^{-3}$ & 0.9962 & 1.71 \\
		49 & $1.5\times10^{-4}$ & 381 & $1.42\times10^{-3}$ & $1.44\times10^{-3}$ & 0.9971 & 4.77 \\
		65 & $1.0\times10^{-4}$ & 469 & $9.54\times10^{-4}$ & $1.08\times10^{-3}$ & 0.9981 & 13.84 \\
		\bottomrule
	\end{tabular}
\end{table}

\FloatBarrier
\section{PINN and XPINN Baseline Hyperparameters}
\label{app:pinn_xpinn_hyperparameters}

\begin{figure}[pos=htbp]
	\centering
	\includegraphics[width=0.9\linewidth]{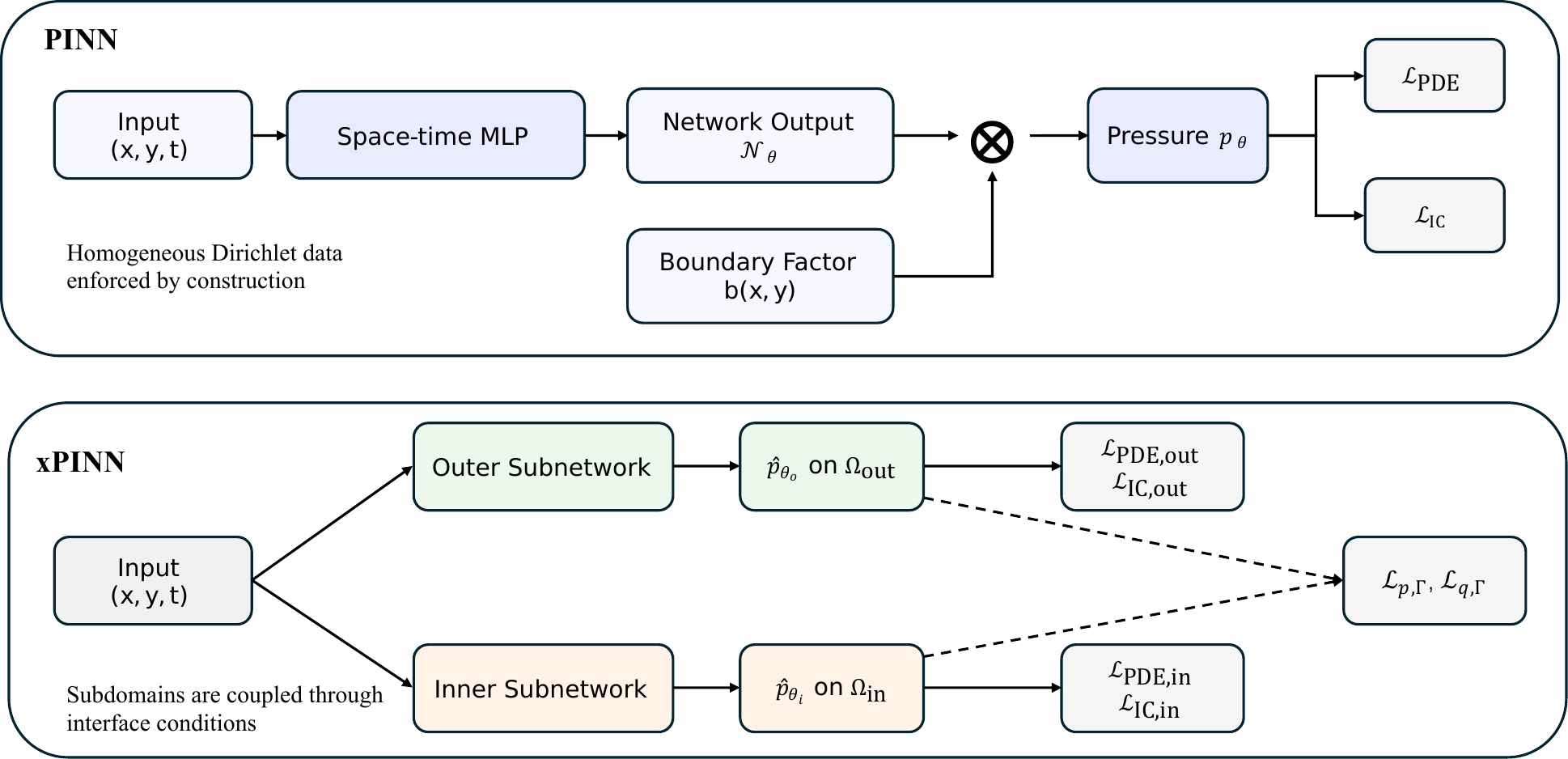}
	\caption{A single space--time PINN and a two-subnetwork XPINN with pressure- and normal-flux-continuity penalties.}
	\label{fig:app_pinn_xpinn_schematic}
\end{figure}

This appendix reports the implementation details of the PINN and XPINN baselines used in Section~\ref{sec:central_low_perm}. Both baselines were evaluated on the central low-permeability benchmark with the same grid, final time, and time step as the finite-volume reference problem. The pressure was represented by a space-time multilayer perceptron with input $(x,y,t)$ and scalar output. Homogeneous Dirichlet boundary conditions were imposed through the hard-boundary factor
\begin{equation}
	b(x,y)=x(1-x)y(1-y),
\end{equation}
so that the predicted pressure was written as $p_\theta(x,y,t)=b(x,y)N_\theta(x,y,t)$. Therefore, no separate boundary-condition loss was used for these two baselines.

\begin{table}[pos=!htbp,abovecap=2pt]
	\centering
	\setlength{\tabcolsep}{5pt} 
	\small
	\caption{Common training settings for the PINN and XPINN baselines on the central low-permeability benchmark.}
	\label{tab:app_pinn_xpinn_common}
	\begin{tabular}{lcc}
		\toprule
		Setting & PINN & XPINN \\
		\midrule
		Grid & $65\times65$ & $65\times65$ \\
		Final time $T$ & $0.05$ & $0.05$ \\
		Reference time step & $10^{-4}$ & $10^{-4}$ \\
		Network input & $(x,y,t)$ & $(x,y,t)$ \\
		Hidden layers & 3 & 3 per subnet \\
		Layer width & 48 & 48 per subnet \\
		Activation & $\tanh$ & $\tanh$ \\
		Trainable params & 4945 & 9890 \\
		Optimizer & Adam & Adam \\
		Adam iterations & 4000 & 4000 \\
		Learning rate & $10^{-3}$ & $10^{-3}$ \\
		Precision & torch.float32 & torch.float32 \\
		Device & CUDA & CUDA \\
		Random seed & 42 & 43 \\
		Weight init. & Xavier normal, zero bias & Xavier normal, zero bias \\
		\bottomrule
	\end{tabular}
\end{table}

For the PINN baseline, the loss function was
\begin{equation}
	\mathcal{L}_{\mathrm{PINN}}
	=
	\lambda_{\mathrm{pde}}\mathcal{L}_{\mathrm{pde}}
	+
	\lambda_{\mathrm{ic}}\mathcal{L}_{\mathrm{ic}},
\end{equation}
where $\mathcal{L}_{\mathrm{pde}}$ was computed from the residual $p_t-A(p_{xx}+p_{yy})$ and $\mathcal{L}_{\mathrm{ic}}$ enforced $p(x,y,0)=\sin(\pi x)\sin(\pi y)$. The permeability was evaluated as the piecewise constant coefficient of the central-inclusion benchmark.

For the XPINN baseline, two subnetworks were used for the outer matrix region and the central inclusion. The outer residual used $A=1$, and the inner residual used $A=0.2$. The loss function additionally included pressure-continuity and normal-flux-continuity terms on the material boundary,
\begin{equation}
	\mathcal{L}_{\mathrm{XPINN}}
	=
	\lambda_{\mathrm{pde,out}}\mathcal{L}_{\mathrm{pde,out}}
	+
	\lambda_{\mathrm{pde,in}}\mathcal{L}_{\mathrm{pde,in}}
	+
	\lambda_{\mathrm{ic}}\left(
	\mathcal{L}_{\mathrm{ic,out}}+\mathcal{L}_{\mathrm{ic,in}}
	\right)
	+
	\lambda_p\mathcal{L}_{p,\Gamma}
	+
	\lambda_q\mathcal{L}_{q,\Gamma}.
\end{equation}
Here $\mathcal{L}_{p,\Gamma}$ penalizes the pressure jump, and $\mathcal{L}_{q,\Gamma}$ penalizes the jump in the normal flux $A\nabla p\cdot n$.

Collocation points were randomly resampled during Adam training.

\begin{table}[pos=!htbp,abovecap=2pt]
	\centering
	\caption{Collocation counts and loss weights for the PINN and XPINN baselines.}
	\label{tab:app_pinn_xpinn_loss}
	\small
	\begin{tabular}{lcc}
		\toprule
		Setting & PINN & XPINN \\
		\midrule
		PDE residual points & 2048 & 1536 outer, 512 inner \\
		Initial-condition points & 2048 & 1536 outer, 512 inner \\
		Interface points & -- & 512 \\
		$\lambda_{\mathrm{pde}}$ & 1.0 & -- \\
		$\lambda_{\mathrm{pde,out}}$ & -- & 1.0 \\
		$\lambda_{\mathrm{pde,in}}$ & -- & 1.0 \\
		$\lambda_{\mathrm{ic}}$ & 10.0 & 10.0 \\
		$\lambda_p$ & -- & 50.0 \\
		$\lambda_q$ & -- & 1.0 \\
		Training-log interval & 200 steps & 200 steps \\
		\bottomrule
	\end{tabular}
\end{table}

\section*{Acknowledgment}
The research of Peiqi Li was funded by the Postgraduate Research Scholarship of
Xi'an Jiaotong-Liverpool University (FOSA2412003).

\printcredits

\section*{Declaration of competing interest}
The authors declare that they have no known competing financial interests or
personal relationships that could have appeared to influence the work reported
in this paper.

\bibliographystyle{cas-model2-names}
\bibliography{cas-refs}

\end{document}